\documentclass[11pt]{article}
\usepackage{amsmath,latexsym,amsxtra,enumerate,
color}
\usepackage{bbm, dsfont}
\usepackage{amsthm}
\usepackage[english]{babel}
\usepackage[dvipsnames,svgnames,table]{xcolor}
\usepackage{times, amsfonts, amssymb, amsthm, amscd, graphicx,stmaryrd}
\usepackage[mathscr]{euscript}
\usepackage{comment}
\usepackage{cancel}

\RequirePackage[colorlinks, linkcolor=blue, citecolor=blue,urlcolor=blue]{hyperref}

\newcommand{\la}{\lambda}
\newcommand{\D}{\mathcal{D}}
\newcommand{\R}{\mathbb{R}}
\newcommand{\N}{\mathbb{N}}
\newcommand{\E}{\mathbb{E}}
\newcommand{\one}{\mathds{1}}
\newcommand{\tx}{{\tt x}}
\newcommand{\f}{ \mathbf{f}}
\newcommand{\Y}{\mathcal{Y}}
\newcommand{\Z}{\mathcal{Z}}
\newcommand{\U}{\mathcal{U}}
\newcommand{\om}{\omega}
\newcommand{\Om}{\Omega}
\newcommand{\ftn}{\mathcal{F}}

\newcommand{\equa}{\begin{eqnarray*}}
\newcommand{\tion}{\end{eqnarray*}}
\newcommand{\equal}{\begin{eqnarray}}
\newcommand{\tionl}{\end{eqnarray}}
\newcommand{\m}{\mathbbm{m}}
\newcommand{\DD}{{\mathbb{D}_{1,2}}}

\newcommand{\dJ}{d_{J_1}}

\newcommand{\EE}{\mathbb{E}}

\newcommand{\NN}{\mathbb{N}}

\newcommand{\PP}{\mathbb{P}}
\newcommand{\QQ}{\mathbb{Q}}
\newcommand{\RR}{\mathbb{R}}

\newcommand{\bB}{\mathcal{B}}

\newcommand{\fF}{\mathcal{F}}
\newcommand{\gG}{\mathcal{G}}

\newcommand{\lL}{\mathcal{L}}
\newcommand{\mM}{\mathcal{M}}
\newcommand{\nN}{\mathcal{N}}

\newcommand{\pP}{\mathcal{P}}

\newcommand{\tT}{\mathcal{T}}

\newcommand{\fX}{\mathfrak{X}}

\newcommand{\fb}{\mathfrak{b}}
\newcommand{\fc}{\mathfrak{c}}

\newcommand{\fs}{\mathfrak{s}}

\newcommand{\tr}{{\tt r}}

\newcommand{\ry}{{\rm y}}
\newcommand{\rz}{{\rm z}}
\newcommand{\ru}{{\rm u}}
\newcommand{\ov}{\overline}

\newcommand{\wf}{\widehat f}
\newcommand{\wh}{\widehat h}

\newcommand{\bx}{{\bf x}}
\newcommand{\by}{{\bf y}}

\theoremstyle{plain}
\newtheorem{thm}{Theorem}[section]
\newtheorem{lem}[thm]{Lemma}
\newtheorem{prop}[thm]{Proposition}

\newtheorem{cor}[thm]{Corollary}
\newtheorem{lemma}[thm]{Lemma}

\theoremstyle{definition}

\newtheorem{definition}[thm]{Definition}

\newtheorem{rem}[thm]{Remark}

\newtheorem{example}[thm]{Example}
\newtheorem{assumption}[thm]{Assumption}

\allowdisplaybreaks

\makeatletter
\def\timenow{\@tempcnta\time
\@tempcntb\@tempcnta
\divide\@tempcntb60
\ifnum10>\@tempcntb0\fi\number\@tempcntb
:\multiply\@tempcntb60
\advance\@tempcnta-\@tempcntb
\ifnum10>\@tempcnta0\fi\number\@tempcnta}
\makeatother

\newcommand{\com}{\color{red}}

\newcommand{\ch}{\color{blue}}

\usepackage[colorinlistoftodos,bordercolor=orange,backgroundcolor=orange!20,linecolor=orange,textsize=scriptsize]{todonotes}
\newcommand{\ar}[1]{\todo[inline]{\textbf{A.: }#1}} 
 
\title{Locally Lipschitz Path Dependent FBSDEs with Unbounded Terminal
 Conditions in Brownian and L\'evy Settings}

\author{Hannah Geiss$^1$  \hspace{0.5em}   C\'eline Labart$^2$  \hspace{0.5em}   Adrien Richou$^{3,4}$ \hspace{0.5em}   Alexander Steinicke$^5$\\
   \small  \today
}
\date{}

\begin{document}

\maketitle
\begin{abstract}
This paper is dedicated to the analysis of forward backward stochastic
differential equations driven by a Lévy process. We assume that the generator and the terminal condition are path-dependent and satisfy a local Lipschitz condition. We study solvability and Malliavin differentiability of such BSDEs.\\

The proof of the existence and uniqueness is done in three steps. First of all, we truncate and localize the terminal condition and the generator. Then we use an iteration argument to get bounds for the solutions of the truncated BSDE (independent from the level of truncation). Finally, we let the level of truncation tend to infinity. A stability result ends the proof. The Malliavin differentiability result
is based on a recent characterisation for the Malliavin Sobolev space $\mathbb{D}_{1,2}$ by S. Geiss and Zhou. 

\end{abstract}

\vspace{1em}
{\noindent \textit{Keywords:}  L\'evy driven BSDEs, path dependent and locally Lipschitz parameters, Malliavin differentiability of BSDEs, existence and uniqueness of solutions to BSDEs
}
{\noindent
\footnotetext[1]{Department of Mathematics and Statistics, University of Jyvaskyla, Finland,\\ \hspace*{1.5em} hannah.r.geiss@jyu.fi}
\footnotetext[2]{CNRS, LAMA, Univ. Savoie Mont Blanc, France, celine.labart@univ-smb.fr  }
\footnotetext[3]{Université de Bordeaux, Institut de Mathématiques de Bordeaux, UMR CNRS 5251, France, \\ \hspace*{1.5em}adrien.richou@math.u-bordeaux.fr }
\footnotetext[4]{The author research has benefited from the support of the ANR Project ReLISCoP (ANR-21-CE40-0001).}
\footnotetext[5]{Department of Mathematics, Technical University of Leoben,  Austria, \\\hspace*{1.5em}  alexander.steinicke{\rm@}unileoben.ac.at}



\section{Introduction}

Backward stochastic differential equations (BSDEs) have been studied for many years, since the seminal papers from Pardoux and Peng (\cite{PP2} in 1990 and \cite{PardouxPeng92} in  1992) that study the Brownian motion setting. Then, this class of equation was extended in the setting of random measures associated to a L\'evy process by Barles, Buckdahn, and Pardoux (\cite{bbp} in 1997). 
Given that BSDEs have numerous applications, for example in optimal stochastic control,  mathematical finance, semilinear PDEs or stochastic differential games, existence and uniqueness of solutions have been a central topic, as has Malliavin differentiability, the connection to semilinear PDEs 
for forward backward SDEs, and numerical simulation. We refer to \cite{Zhang17} for an extensive presentation of results and applications in the Brownian setting and to \cite{Papapantoleon-18} for a comprehensive literature review concerning the jump framework.

\paragraph{}
To match the settings in applications, weaker assumptions on the coefficients are often needed. For example, quadratic or super-quadratic behavior in the control variable instead of Lipschitz continuity is often needed in optimal stochastic control applications. 
In the Brownian setting, the quadratic case was deeply studied since the seminal paper \cite{Koby00}, and it is known that it requires rather strong necessary assumptions on the terminal condition and the random part of the generator: in \cite{BriandHu06} and \cite{Delbaen-11} existence and uniqueness were shown if they satisfy an exponential moment condition. We also refer to \cite{Bahlali} where a method to prove existence of solutions to quadratic BSDEs with unbounded terminal conditions using comparison theorems is developed. On the contrary, the super-quadratic setting is much less studied. First results were obtained in \cite{Delbaen-11b} for the bounded case and later in \cite{CheriditoNam} for a framework where the control variable is bounded. Then, the unbounded framework (for the terminal condition, the random part of the generator and the control variable) was tackled in \cite{Richou12}.

In the L\'evy setting, the quadratic-exponential framework - this designation coming from the fact that it allows an exponential growth with respect to the jump part of the control variable - was strongly studied due to some applications to exponential utility maximization problems, see e.g. \cite{Bech, Morlais-10, Kazi-TaniPossamaiZhou15, BeBuKe} or \cite{FujiTaka} in the bounded case. The unbounded case was studied in \cite{AntoManc16} and in the preprints \cite{ElKaroui16} and \cite{MatoussiSalhi19}. Still in the unbounded case we also refer to the preprint \cite{MatoussiSalhi20} for the reflected setting and to \cite{GuLinXu24} for the marked point processes setting. As in the Brownian setting, the super-quadratic framework  gets much less attention: see \cite{GeissStein20}.

\paragraph{}
In the present paper we consider the FBSDE  given by \eqref{forward} and \eqref{fbsde} below. These equations are driven by a Brownian motion and/or  an independent compensated Poisson random measure.  We treat a path-dependent framework (for the SDE's drift and the dependence of the FBSDE on the SDE) but we can emphasize that our setting includes the Markovian case: see {Remark \ref{rem:SDE-assumptions}} and Example \ref{Examples}.

Our main result, Theorem \ref{ex-and-unique-sol}, requires the Assumptions
\ref{X-assumptions} and \ref{Ypath-assumptions}  and  states existence and uniqueness of the solutions which are bounded by an expression involving the  $\sup$-norm of the forward process. It also states that the solutions are Malliavin differentiable. In order to better reflect the novelty of our results, it is important to put our assumptions into perspective with regard to the current state of the art.
\begin{itemize}
\item If $\tt r=0$ ($\tt r$ is introduced in Assumption \ref{Ypath-assumptions}), our results generalize \cite{GeissStein20} since we do not need any boundedness on the terminal condition and the random part of the generator. It also extends some Brownian framework results of \cite{CheriditoNam} to the path dependent setting.
\item In the pure Brownian setting, we do not cover results of \cite{Richou12} since we need to assume $0 \leq {\tt r} \leq \frac{1}{2\ell}$, while in \cite{Richou12}  we only need $0 \leq {\tt r} \leq \frac{1}{\ell}$. This limitation seems strange at first sight since our proof strategy is inspired by \cite{Richou12} but can be easily explained. Indeed, the technical Lemma \ref{exponential-bound} which is essential in our proof gives us that 
$$\mathbb{E}[e^{c\sup_{0 \leq t \leq T} |X_t|^p}]<+\infty$$
for all $c > 0$ and all $p \leq 1$. It is easy to show that this estimate is not true for $p>1$ by considering a standard Poisson process. Nevertheless, in the Brownian setting this estimate can be extended to $p=2$ and $c$ small enough: remark that this is obvious for the standard Brownian motion and see e.g. part 5 in \cite{BriandHu08} for the general case.
\item In the quadratic framework ($\ell=1$), we have an existence and uniqueness result in some cases where the terminal condition is not exponentially integrable. Indeed, by taking ${\tt r}=1/2$ and $|X_T|^{3/2}$ for the terminal condition, non purely Brownian L\'evy processes satisfy $\mathbb{E}\left[e^{c|X_T|^{3/2}}\right]=+\infty$ for all $c>0$. 
This is not in contradiction with known results for unbounded quadratic-exponential BSDEs where the exponential integrability is asked and cannot be relaxed for the classical entropic risk measure since our framework does not cover the quadratic-exponential setting.
\item We assume in Assumptions \ref{X-assumptions} that $\sigma$ and $\rho$ are deterministic. This strong assumption is essential in our proof and cannot be relaxed easily to a setting where $\sigma$ and $\rho$ are functions of the process (and even less for functions of the process path). Indeed, in the Brownian setting we can show that the estimate \eqref{bd-cond-x} for $Z$ is still valid in the quadratic framework, for the bounded setting or ${\tt r}=0$: see \cite{Richou12} and \cite{BriandRichou} where BMO tools are essential. Let us remark that in the quadratic framework, the estimates \eqref{bd-cond-x} are interesting for some applications, for example analysis of numerical schemes, but are not useful for proving existence and uniqueness results. Finally, up to our knowledge, there is nothing known in the super-quadratic framework where we do not have BMO tools on hand. In order to preserve the readability of this article and since only the quadratic framework could be tackled, we will not treat here the extension of \cite{Richou12} and \cite{BriandRichou} to the L\'evy setting when $\sigma$ and $\rho$ are not deterministic.
\item For the  Malliavin differentiability of SDEs there is a very general result in \cite{ImkellerReisSalkeld}. However, it does not cover our specific  path-dependent setting which permits  a pointwise   estimate for the  Malliavin derivative. Our result on Malliavin differentiability of  BSDEs with jumps generalizes  \cite{GeissStein16} and \cite{FujiTaka} to the unbounded case.

\end{itemize}

\paragraph{}
Let us now give some details concerning the proof strategy and tools used.

The starting point of the proof of the existence and uniqueness result is to truncate/localize the terminal condition $g$ and the generator $(\bx,z,u) \mapsto f(.,\bx,.,z,u)$ in order to come across the classical Lipschitz framework. Then, a crucial step is given by Proposition \ref{bounds-independend-from-M} where, as in \cite{Richou12}, an iteration argument is applied to get bounds for the solution of the truncated BSDE, independent from the level of truncation. Finally, we make the truncation level tend to infinity and we conclude thanks to a stability result given by Lemma \ref{lem:apriori}.

Concerning the Malliavin differentiability result, a recent characterisation for the Malliavin Sobolev space $\DD$  from \cite{GeissZhou} (see  Theorem \ref{StefanXilin} below) enables us to get Malliavin differentiability in the Gaussian setting of the path-dependent SDE and BSDE very easily. Another characterisation of  $\DD$ goes back to Sugita (\cite{Sugita}) and was rediscovered in the recent years as a useful tool (\cite{GeissStein16,ImkellerReisSalkeld,Mast-etal}). In combination with results from Janson \cite{Janson} we use it here for instance to estimate Malliavin derivatives of path dependent functions in the Gaussian setting or in the 'Gaussian direction' of the derivative in the L\'evy setting. Let us emphasize that these estimates are crucial to get bounds on the truncated solution in Proposition \ref{bounds-independend-from-M}.

\paragraph{}
Eventually, we highlight that the estimates \eqref{bd-cond-x}, in particular those on $Z$ and $U$, can be used for tackling interesting applications. For example, a comparison result can be derived easily thanks to \eqref{bd-cond-x} and the classical linearization technique. Then, this comparison result can be used to prove that the unique solution of our FBSDE provides a viscosity solution to the associated integro-PDE. It could be also possible to prove the (classical) differentiability of the solution under appropriate assumptions, by following same ideas as in \cite{Masiero14}. Finally, the estimates \eqref{bd-cond-x} allow to get an upper bound on the truncation error between the original BSDE and the one where the dependence of the generator with respect to $(z,u)$ is bounded in order to get a classical Lipschitz generator. This kind of result is important if we want to have a numerical scheme for solving the BSDE: Indeed, a natural approach is to use a classical scheme on the Lipschitz approximated BSDE (e.g. \cite{BouchardElie08}) and then upper bound the global error thanks to the truncation error and the discretization error for the truncated BSDE. Obviously, the two terms depend strongly on the truncation level. Nevertheless, in the Brownian setting it is possible to set it in a way such that the global error is reasonable, see \cite{Richou12}, and it should be possible to use the same in the L\'evy setting.

\paragraph{}
In the remaining, the paper is organized as follows: Section \ref{setting} contains necessary assumptions and notations in order to understand the underlying problem. Section \ref{sec: main} presents our main result on existence, uniqueness and Malliavin differentiability of superquadratic BSDEs together with remarks, explanations and examples. The sections following directly afterwards are preliminaries for the proof of the main result but convey interesting results themselves: In Section \ref{sec: preliminaries} we show that the forward process admits all exponential moments and Section \ref{Mall} presents tools for Malliavin calculus in the L\'evy setting. Finally, in Section \ref{sec: bounds} bounds for the BSDE's solution processes are obtained and in Section \ref{sec: finish}, together with an a priori-estimate, the main result's proof is finalized. Some auxiliary results are postponed into the \ref{app}ppendix.

\section{Setting and notations \label{setting}}

Let $\mathcal{L}=\left( \mathcal{L}_t\right)_{t\in{[0,T]}}$  (for $T>0$ fixed)  be a c\`adl\`ag L\'evy process  on a complete probability space $(\Omega,\mathcal{F},\mathbb{P})$ with L\'evy measure $\nu$ on $\R_0:= \R\!\setminus\!\{0\}$ (especially, also $\nu\equiv 0$ is permitted). 
We will denote the augmented natural filtration of $\mathcal{L}$ by
$\left({\mathcal{F}_t}\right)_{t\in{[0,T]}}$ and assume that $\mathcal{F}=\mathcal{F}_T.$ \\
 \bigskip
 The L\'evy-It\^o decomposition of a L\'evy process $ \mathcal{L}$ can be written as
\equa
\mathcal{L}_t = \text{b} t + \varsigma W_t   +  \int_{{]0,t]}\times \{ |x|\le1\}} x\tilde{\nN}(ds,dx) +  \int_{{]0,t]}\times \{ |x|> 1\}} x  \nN(ds,dx),
\tion
where $\text{b}\in \R, \varsigma\geq 0$, $W$ is a Brownian motion and $\nN$ ($\tilde \nN$) is the (compensated) Poisson random measure associated with $\mathcal{L}$.

\subsection{Notation for the Skorohod space}
\begin{itemize}
\item 
Let $D{[0,T]}$ denote the space of c\`adl\`ag functions on the interval ${[0,T]}$. For $\bx \in   D[0,T]$ let $|\bx |_\infty := \sup_{0\le t\le T} |\bx_t|.$  
\item We equip $D[0,T]$ with  the  (modified) Skorokhod $J_1$-metric $d_{J_1}$,  (see \cite[Section 3.2]{Kern}) given by
\begin{align*}
\dJ(\bx,\by)=\inf_\lambda\bigg\{|\bx\circ\lambda-\by|_\infty {\ch +} \sup_{s\neq t}\Big|\log\Big(\frac{\lambda(s)-\lambda(t)}{s-t}\Big)\Big|\bigg\},
\end{align*}
where the infimum is taken over all increasing bijections $\lambda\colon [0,T]\to[0,T]$. Note that $\dJ(\bx,\by)\leq|\bx-\by|_\infty$.
For properties of $(D{[0,T]},d_{J_1})$ and further reading see \cite{Bill,Kern,Skorohod,Whitt}, for instance.

\item Let
$\mathcal{B}(D{[0,T]})$  be the Borel $\sigma$-algebra  generated by the open sets of $D{[0,T]}.$  This   $\sigma$-algebra coincides with the $\sigma$-algebra generated by the  coordinate projections   
from $ D{[0,T]}$ to  $\R$ given by  $\ \bx  \mapsto \bx_t,\  t\in [0,T]$ (see \cite[Theorem 12.5]{Bill}). 

\item  For  $ \bx\in D{[0,T]}$ and  $t\in [0,T]$  we set 
$$ \bx^t:=(\bx_{t\wedge s})_{s\in{[0,T]}}$$  and define 
$$D{[0,t]}:=\left\{\bx\in D{[0,T]} : \bx^t=\bx \right\}.$$ By this identification
we define a filtration on this space by
\equal \label{filtrationG-t}
\mathcal{G}_t:=\sigma\left(\mathcal{B}\left(D{[0,t]}\right)\cup \operatorname{N}_{\mathcal{L}}{[0,T]}\right), \quad 0\leq t\leq T,
\tionl
where $\operatorname{N}_{ \mathcal{L}}{[0,T]}$ denotes the $\mathbb{P}_{\mathcal{L}}$-
null sets of $\mathcal{B}\left(D{[0,T]}\right)$ where $\mathbb{P}_{\mathcal{L}}$ is the  image measure of the L\'evy process ${\mathcal{L}}\colon\Omega\to D{[0,T]},\omega \mapsto  { \mathcal{L}}(\omega).$ 
\end{itemize}

 \subsection{Further spaces and notations}

\begin{itemize}

\item For $p> 0$, a measure space $(M,\Sigma,\mu)$ and a Banach space $\mathrm{E}$, as usual, let
\begin{align*}
&L^p(M,\Sigma,\mu;\mathrm{E})\\&:=\Big\{f\colon M \rightarrow E : f\text{ is measurable and }\int_M\|f\|_{\mathrm{E}}^p d\mu<\infty\Big\}\big/ N_\mu,
\end{align*}
where $N_\mu$ are the measurable functions that are zero $\mu$-a.e. For the boundary cases, $L^0(M,\Sigma,\mu;\mathrm{E})$ is the space of measurable functions $f\colon M\to E$ (up to a function being zero $\mu$-a.e.) and $L^\infty(M,\Sigma,\mu;\mathrm{E})$ is the space of measurable functions such that $\|f\|_E$ is essentially bounded (again up to $\mu$-a.e.~zero functions). Slightly inaccurate, depending on use, when taking a function out of such a space, we mean to work with a representative of an equivalence class.

Further, if the underlying measure space and/or $\sigma$-algebra is clear, we will just drop them in the notation, if further $\mathrm{E}=\RR$, we just write $L^p(\mu)$. 

For the particular case $\mathrm{E}=\RR$, we let $L^p([0,T]):=L^p([0,T],\mathcal{B}([0,T]), \lambda),$ where $\lambda$ stands for the Lebesgue measure. 
\item For $t\in [0,T]$ and a probability measure $\QQ$ on $(\Omega, \fF)$  we denote 
$$\EE^\QQ_t \cdot = \EE^\QQ[\cdot |\fF_t] \quad \text{and} \quad  \EE^\QQ_{t-} \cdot = \EE^\QQ \Big[\cdot \Big | \sigma \Big (\bigcup_{s<t }  \fF_s \Big) \Big]. $$

\item We denote by $\operatorname{Prog}$ the $\sigma$-algebra of $(\fF_t)_t$-progressively measurable sets on $\Omega\times[0,T]$ and by $\pP$ the one of $(\fF_t)_t$-predictable sets generated
by the left-continuous $(\mathcal{F}_t)$-adapted processes thereon.
\item For  $1\le p \le \infty$ let  $\mathcal{S}^p$ denote the  space of all $\operatorname{Prog}$-measurable and c\`adl\`ag processes  $Y\colon\Omega\times{[0,T]} \rightarrow \R$ such that
\equa
\left\|Y\right\|_{\mathcal{S}^p}:=\||Y|_\infty \|_{L^p} <\infty.
\tion
\item  $L^p(W) $ denotes the space of all $\operatorname{Prog}$-measurable processes $Z\colon \Omega\times{[0,T]}\rightarrow \R$  such that
\equa
 \|Z \|_{L^p( \PP\otimes\la)} <\infty.
\tion
\item We define $L^2(\tilde \nN)$ as the space of all random fields $U\colon \Omega\times{[0,T]}\times{\R_0}\rightarrow \R$
which are measurable with respect to
$\mathcal{P}\otimes\mathcal{B}(\R_0)$ such that
\equa
\left\|U\right\|_{L^2(\tilde \nN) }^2:=\E\int_{{[0,T]}\times{\R_0}}\left|U_s(x)\right|^2 ds \nu(dx)<\infty,
\tion
\end{itemize}

\section{The main result}\label{sec: main}

We consider for $t \in [0,T]$
\begin{align} \label{forward}
X_t = \, & x + \int_0^t b(s, (X_r^{s})_{r\in [0,T]}) ds +  \int_0^t \sigma(s) dW_s + \int_{{]0,t]}\times{\R_0}}  \rho(s,v)  \tilde \nN(ds,dv), \\
 \label{fbsde} Y_t= \,&g((X_s)_{s\in [0,T]})+\int_t^T  f\left( s, (X_r^{s})_{r\in [0,T]},Y_s, Z_s,  H_s  \right)ds  \notag \\
& -     \int_t^T Z_{s}   dW_s -\int_{{]t,T]}\times{\R_0}}U_{s}(v) \tilde \nN(ds,dv), 
\end{align}
with   $H_s := \int_{\R_0}  h(s, U_s(v)) \kappa (v)  \nu(dv),$ and 
$\kappa (v) := 1\wedge |v|.$ Recall that $X_r^s = X_{s \wedge r}$}.

The structure of the dependency on $U$ by a $\nu(dv)$-integral functional is an explicit way to employ a locally Lipschitz dependence in this variable. With this certain form, we can rely on earlier BSDE results, e.g.~those in \cite{DelongImkeller10, FujiTaka, GeissStein16}, especially concerning existence and comparison theorems. 

\begin{definition}
  We say that $(Y,Z,U)$ is a solution to \eqref{fbsde} if the triplet of processes satisfies \eqref{fbsde} and belongs to $\mathcal{S}^2\times L^2(W)\times L^2(\tilde{\nN})$.
\end{definition}

We agree on the following assumptions on the coefficients.

\begin{assumption}[for $X$] \label{X-assumptions} 
The functions $b: [0,T]\times D[0,T] \to \R,$ $\sigma: [0,T]\to \R,$
and  $\rho: [0,T] \times \R \to \R$  are jointly measurable 
and satisfy
\begin{enumerate}[(1)]
\item  $|b(t,0)| \le K_b$, $|b(t,\bx)-b(t,\bx')|\le L_b |\bx-\bx'|_\infty,$ for all $(t,\bx)\in [0,T] \times D[0,T],$\label{ass:sdeb}
\item \label{bounded-sigma} $|\sigma(t)| \le K_\sigma,$  for all $t\in [0,T]$,
\item \label{bounded-gamma} $|\rho(t,z)| \le \kappa_{\rho}(z)$ for all $(t,z)\in [0,T] \times \R$  with $ \kappa_{\rho} \in L^2(\nu)\cap L^\infty(\nu) $. \\
We denote $\kappa_{\rho,2}:=\|\kappa_{\rho}\|_{L^2(\nu)}$ and  $\kappa_{\rho,\infty}:=\|\kappa_{\rho}\|_{L^{\infty}(\nu)}$.
\end{enumerate}
\end{assumption}

\begin{rem}\label{rem:SDE-assumptions}

\begin{itemize}\hfill
\item Note that in Assumption \ref{X-assumptions}\eqref{ass:sdeb}, even if there is no dependency on $t$, measurability w.r.t.~$\bB(D[0,T])$ needs to be required, as it does not follow from continuity w.r.t.~$|\cdot|_\infty$ (in contrast to the situation on the space of continuous functions $C[0,T]$). See the counterexample in Remark \ref{the-impossible}.
\item The mapping $b\colon (t,\bx)\mapsto \tilde b(t, \bx_T)$ for a measurable function $\tilde{b}\colon [0,T]\times\R\to\R$, which is Lipschitz in the second variable fits our setting. Since 
$$\tilde{b}(t,X_t)=\tilde{b}(t,X^t_T)=b(t,\bx^t),$$
the Markovian setting is included in ours.
\item The SDE \eqref{forward} has a unique solution under these assumptions. This can be seen following the proof in \cite[V.2.8-12]{RW00}, where Picard iterations are constructed that converge to a solution. For another approach in the continuous setting, see \cite[Chapter IX,\S 2]{RY99}.
Exactly the same as in \cite{RW00} can be done in our situation (even more generally, with an additional Lipschitz dependency in $\sigma$ and $\varrho$). The limit of the Picard iterations in $|\cdot|_\infty$ is also a limit in $\dJ$ (the coarser topology), therefore, measurability in $D[0,T]$ pertains.

\end{itemize}

\end{rem}

\smallskip

\begin{assumption}[for $(Y,Z,U)$]\label{Ypath-assumptions}
For a  measurable function
$$ f: [0,T]\times D{[0,T]}  \times \R \times \R \times \R \to \R $$
there exists a  function   $k_f   \in L^1([0,T])$ 
such that $\forall t \in [0,T]$   it holds
\equa 
&&\hspace{-4em}    |f(t,0,0,0,0)|  \le   k_f(t). 
 \tion

Moreover, we assume that there exist $c>0$,  $\ell \ge 1,  0\le {\tt r} \le \frac{1}{2\ell},   \alpha\ge0, \beta\ge 0, \gamma\ge 0,$  $L_{f,\ry}\ge 0 $ and
$ m_1, m_2 \ge 0$,  $ m_1+ m_1m_2+m_2 \le \ell$ such that 
\begin{enumerate}[(i)]
\item    $g: D{[0,T]} \to \R$ is $\bB(D[0,T])-\bB(\R)$-measurable and \\
 $ |g({\bf x})-g({\bf x}') | \le (c+  \frac{\alpha}{2} (|\bx|^{\tr}_\infty+ |\bx'|^{\tt r}_\infty))|\bx-\bx'|_\infty$ \label{the-g}, 

\item 
for $f$ we assume, that for almost all $s$,\\
$|f(s,\bx, y, z, u)-f(s,\bx', y,z,u) | \le (c+ \frac{\beta}{2}(|\bx|^ {\tt r}_\infty + |\bx'|^ {\tt r}_\infty))|\bx-\bx'|_\infty,$
\item \label{y-ass} $ |f(s,\bx, y, z, u)-f(s,\bx, y', z, u) | \le L_{f,\ry} |y- y'| $,
\item $ |f(s,\bx,y, z, u)-f(s,\bx, y, z', u) | \le \left (c+ \frac{\gamma}{2} (|z|^\ell + |\rz'|^\ell) \right )| z- z'|$,
\item $ |f(s,\bx, y, z, u)-f(s,\bx, y, z,u') | \le  \left(c+ \frac{\gamma}{2} (| u|^{m_1} + |u'|^{m_1}) \right )|u-u'| $,
\item \label{h-assumption} $h: [0,T]\times \R  \to \R$ is continuous  and satisfies  \\
$ |h(s,u)-h(s,u') | \le  \left (c+ \frac{\gamma}{2} (|u|^{m_2}+ |u'|^{m_2}) \right )|u-u'| , \quad h(s,0)=0, $
\item \label{fuhugeminus1} $ f_\ru(s,\bx, y, z, u) h_\ru(s,  u')   \ge -1$ where  $f_\ru$  and $h_\ru$  stand for the weak derivatives.
\end{enumerate}
\end{assumption}

Our main result is the following theorem.

  \begin{thm} \label{ex-and-unique-sol}
Let  $(g,f,h)$ be such that Assumptions  \ref{X-assumptions} and \ref{Ypath-assumptions} are  satisfied. Then, there exists a solution $(Y,Z,U)$ such that
$$\EE|Y|_\infty^2+\EE\int_0^T|Z_s|^2ds +\EE\int_0^T \|U_s\|_{L^2(\nu)}^2ds<\infty$$
and there exist   $a,b,c>0$ such that 
 \begin{align}\label{bd-cond-x}
 |Y_t | \le c(1 +  |X^t|^{{\tt r}+1}_\infty),   \,\,\,\ |Z_t | \le a + b |X^t|^{\tt r} _\infty  \,\,\,\text{and}\,\,\, |U_t(v) | \le \kappa_{\rho}(v)(a + b |X^t|^{\tt r}_\infty ).
\end{align}
 Moreover, the solution is unique among the set of solutions satisfying \eqref{bd-cond-x},
and the processes are Malliavin differentiable i.e.
\equa
Y, Z   \in L^2([0,T];\DD), \quad U\in L^2([0,T]\times\R_0;\DD)
\tion
 and for $t \le u \le T$ 
\begin{align}\label{eqn_DZ_DU} 
D_{t,x} Y_u= \,& D_{t,x}g(X)+   D_{t,x} \int_u^T f \left( s, X^s,  Y_s,Z_s,H_s  \right)ds  \notag \\
& -     \int_u^TD_{t,x} Z_{s}   dW_s -\int_{{]u,T]}\times{\R_0}}  D_{t,x} U_{s}(v) \tilde \nN(ds,dv).
\end{align}

We also have the following bounds for $x\neq 0$ and for a.a. $0 \le s \le t\le T$
\begin{align} \label{bound_DZ_DU}
      |D_{s,x}Z_t| \le  C(1+ a +  b |X^t|_\infty^{\tr}),\;\mbox{ and } \;|D_{s,x} U_t(v)  | \le \kappa_{\rho}(v) C(1+ a +  b |X^t|_\infty^{\tr})
\end{align}
where $C$ depends on $\kappa_{\rho}$ through $\kappa_{\rho,\infty}$ and $\kappa_{\rho,2}$.
\end{thm}

\begin{rem} \label{the-impossible} 

\begin{enumerate}
\item 
A function  $g: D{[0,T]} \to \R$ satisfying only
\begin{align} \label{uniform-lip}
|g({\mathbf x})-g({\mathbf y})| \le L_g |{\bf x} -{\bf y}|_\infty
\end{align}
is in general not   or $\gG_T$- and hence also not $\mathcal{B}(D{[0,T]})$-measurable:

In the case of $\bB([0,T])$, consider a set A which is closed w.r.t.~the $\sup$ norm but not Borel measurable w.r.t.~to the Skorokhod metric. 
Such a set must exist because otherwise the Borel $\sigma$-algebras generated by the  $\sup$ norm  and Skorokhod metric would coincide which is not the case - proved in \cite[Section 15]{Bill}, for example.  For a general L\'evy process such a set can also be found if we take the completion $\gG_T$ instead of $\bB([0,T])$. Take for example $\lL$ to be a standard Poisson process. Then, consider the map 
$$\phi\colon ([0,T],\mathrm{Leb}([0,T]))\to (D[0,T],\gG_T,\PP_\lL),\quad t\mapsto \one_{[t,T]},$$
where $\mathrm{Leb}([0,T])$ denotes the Lebesgue-$\sigma$-algebra on $[0,T]$.
The same map with both sigma algebras restricted to $\bB([0,T])$ and $\bB(D[0,T])$ is measurable, see \cite[Section 15]{Bill}. Each set in $\gG_T$ can be written as the symmetric difference $B\triangle M$ with $B\in \bB(D[0,T])$ and $M\in N_\lL[0,T]$. Hence, the preimage $\phi^{-1}(B\triangle M)$ is given by $\phi^{-1}(B)\triangle\phi^{-1}(M)$. The set $\phi^{-1}(B)$ is measurable and $\phi^{-1}(M)$ is a subset of a Lebesgue null set in $[0,T]$ (the distribution of the jump position of $\lL$ in $[0,T]$ and the Lebesgue measure are equivalent) and therefore contained in $\mathrm{Leb}([0,T])$. Thus $\phi$ is also $\mathrm{Leb}([0,T])-\gG_T$-measurable. Take a non-Lebesgue-measurable set $H\subseteq[0,T]$ and set $U:=\bigcup_{t\in H} \{\by\in D[0,T]: |\phi(t)-\by|_\infty<\tfrac{1}{2}\}$. Then, $U$ is open w.r.t.~$|\cdot|_\infty$ and $\phi^{-1}(U)=H\notin \mathrm{Leb}([0,T])$. Therefore, $U$ is not $\gG_T$-measurable and neither is its complement $A:=D[0,T]\setminus U$.
Define the map 
$$g({\bx}) = \inf_{ {\bf y} \in A}  |{\bx} - {\bf y}|_\infty$$
which satisfies \eqref{uniform-lip}. But we have the pre-image $g^{-1}(\{0\}) =A$, hence $g$ is not $\gG_T$ -measurable.

\item We need   $\kappa (x)  := 1 \wedge |x|$   and \eqref{fuhugeminus1} because we want  the  Dol\'eans-Dade exponentials given in  \eqref{the-process-G} to be non-negative, and for a comparison theorem to hold which is applied in the proof of  Proposition \ref{Malliavin-diff-trunc-BSDE} below.

\end{enumerate}

\end{rem}

\begin{example}
\label{Examples} {\color{white}.}
\begin{enumerate}[(a)]
\item Our setting covers the standard Markovian setting. Functions of type
\begin{align*}
g(\bx)=\tilde{g}(\bx_T),\quad f(s,\bx^s,s,y,z,u)=\tilde{f}(s,\bx^s_s,y,z,u)=\tilde{f}(s,\bx^s_T,y,z,u),
\end{align*}
where $\tilde{g}$ and $\tilde{f}$ are (locally) Lipschitz in the second variable w.r.t.~$\R$ meet our conditions: by the (local) Lipschitz condition, they are (locally) Lipschitz w.r.t.~to $|\cdot|_\infty$. Since the projection $\bx\mapsto \bx_T$ is measurable, the requirements are met. 
\item The supremum functional $\bx\mapsto|\bx|_\infty$ fits in our setting: Clearly it is Lipschitz w.r.t.~itself. For the continuity and hence, measurability w.r.t.~$\dJ$ assume a sequence $(\bx_n)_{n\geq 1}$ with $\dJ(\bx_n,\bx)\to 0$, which means that there is a sequence of time distortions $(\lambda_n)_{n\geq 1}$ such that $|\lambda_n-\mathrm{id}|+|\bx_n\circ \lambda_n-\bx|_\infty\to 0$. Since for all admissible $\lambda$ and $\bx \in D[0,T]$, we have that $|\bx|_\infty=|\bx\circ\lambda|_\infty$, it follows that 
\begin{align*}
\big||\bx_n|_\infty-|\bx|_\infty\big|=\big||\bx_n\circ\lambda_n|_\infty-|\bx|_\infty\big|\leq|\bx_n\circ\lambda_n-\bx|_\infty,
\end{align*}
from which continuity follows.

\item In a similar way, the (signed) size of the 'first' jump 
$$\Delta \bx_\tau:=\begin{cases}\bx_{\tau}-\bx_{\tau-},& \tau:=\inf\{t\in [0,T]: |\bx_t-\bx_{t-}|>0\}>0,\\
 0,& \tau=0\text{ or }\tau=\inf\emptyset,\end{cases}$$ of a trajectory $\bx$ is Lipschitz in the $|\cdot|_\infty$-norm with constant 2 and continuous w.r.t~$\dJ$. Continuity, therefore measurability, holds, since the size of the first jump does not change by time shifts $\lambda$, and the idea from the previous example can be carried out again. This works also for other functionals, not depending on a time distortion $\lambda$, as the following example shows too.

\item The maximal jump $j(\bx)$ of a trajectory $\bx$ satisfies our assumptions for $g$: it is Lipschitz with constant $2$ w.r.t.~$|\cdot|_\infty$ and continuous w.r.t.~$\dJ$ (see \cite[Example 12.1]{Bill}).
\item The jump size at a fixed position $s\in [0,T]$, $\bx\mapsto \Delta\bx_s$ is also Lipschitz with respect to $|\cdot|_\infty$. It is measurable (since it is a limit of differences of projections), but, in contrast to the previous example, not continuous with respect to $\dJ$. Take for example $\bx_n=\one_{[\frac{1}{2}-\frac{1}{n},1]}$ and $\bx=\one_{[\frac{1}{2},1]}$. Then $\dJ(\bx_n,\bx)\to0$ but $|\Delta{\bx_{n,}}_\frac{1}{2}-\Delta\bx_\frac{1}{2}|=1$.
\item The integration functional $g(\bx) = \int_0^T \bx_tdt$ satisfies our assumptions, since it is Lipschitz w.r.t.~$|\cdot|_\infty$ (it is even continuous w.r.t.~$\dJ$). Another admissible type of functional is the point evaluation $g(\bx)=\bx_{t^*}$, with $t^*\in [0,T]$ (in this case it is not continuous w.r.t.~$\dJ$). Both types of functionals are special cases of the next item.

\item A continuous linear mapping $\tT\colon(D[0,T],|\cdot|_\infty)\to\R$ has the form
\begin{align*}
\tT(\bx)=\int_{[0,T]}\bx_t dm(t)+\sum_{t\in [0,T]}\Delta\bx_tM(t),
\end{align*}
where $m$ is a finite signed Borel measure and $M\colon[0,1]\to\R$ is such that $\sum_{t\in[0,T]}|M(t)|<\infty$.
This decomposition is from \cite{Pest95}, where the measurability of $\tT$ w.r.t.~$(D[0,T],\bB(D[0,T]))$ is shown too. Therefore, it satisfies our assumptions. Note that when inserting the solution process $X$, the second summand containing $M$ is zero a.s.~since our jumps stem from a L\'evy process. 

\item By the above example, also $\bx\mapsto G_1\bigg(\int_0^TG_2(\bx_t)dm(t)\bigg)$ for Lipschitz real functions $G_1,G_2$ and $m$ as in the previous example is a functional matching our conditions.

\item All the above examples that are formulated in $g(\bx)$ can naturally be turned into functions of type $f(s,\bx,y,z,u)$ having a similar structure in the $\bx$-variable.
\end{enumerate}
\end{example}

\section{Exponential moments of the forward process}\label{sec: preliminaries}

Under our assumptions we have exponential moments for $|X|_\infty$. 

\begin{lem}  \label{exponential-bound}   If Assumption  \ref{X-assumptions} holds then for all $c>0$ the solution $(X_t)_{t\in [0,T]}$ to \eqref{forward} satisfies
$$\E \exp(  c|X|_\infty ) < \infty.$$
 \end{lem}
 
\begin{proof}
We first investigate $\E \sup_{0\leq t\leq T}e^{cX_t}$ for all $c>0$. As   $\E \sup_{0\leq t\leq T}e^{-cX_t}$ can be treated the same way, the finiteness for the modulus follows. 

It\^o's formula (see \cite[Theorem 4.4.7]{Applebaum}; note that  \cite[Assumption (4.14)]{Applebaum} is satisfied  thanks to our Assumption  \ref{X-assumptions} \eqref{bounded-gamma}) yields
\begin{align*}
&e^{c{X}_t}=e^{cx}+\int_0^t e^{c{X}_{s}}\Big(cb(s,X^{s})+\frac{c^2}{2}\sigma(s)^2\Big)ds+\int_0^te^{c{X}_{s}}c\sigma(s)dW_s\\
&\quad+\int_{]0,t]\times\R_0}\bigg(e^{c{X}_{s-}+c\rho(s,z)}-e^{c{X}_{s-}}   \bigg) \tilde{\nN}(ds, dz)\\
&\quad+\int_0^t\int_{\R_0}\bigg(e^{c{X}_{s-}+c\rho(s,z)}-e^{c{X}_{s-}}-c \rho(s,z)e^{c{X}_{s-}}\bigg)\nu(dz)ds\\
&=e^{cx}+\int_0^t e^{c{X}_{s}}\Big(cb(s, X^{s})+\frac{c^2}{2}\sigma(s)^2\Big)ds+\int_0^te^{c{X}_{s}}c\sigma(s)dW_s\\
&\quad+\int_{]0,t]\times\R_0}e^{c{X}_{s-}}\bigg(e^{c\rho(s,z)}-1\bigg) \tilde{\nN}(ds, dz)\\
&\quad+\int_0^t\int_{\R_0}e^{c{X}_{s-}}\bigg(e^{c \rho(s,z)}-1-c \rho(s,z)\bigg)\nu(dz)ds.
\end{align*}
So, the process $\fX_t=e^{c{X}_t}$ solves the SDE (for integration w.r.t.~$dt$ and $dW_t$ we may use $\fX_{t}$ instead of $\fX_{t-}$)
\begin{align*}
d\fX_t = \fb(t,\fX_{t})dt+\fs(t,\fX_{t})dW_t+\int_\R\fc(t,\fX_{ t^-},z)\tilde{\nN}(dt,dz),\quad \fX_0=e^{cx}
\end{align*} 
with the coefficients
\begin{align*}
&\fb(t,x)=x\bigg(cb(t,X^t)+\frac{c^2}{2}\sigma(t)^2+\int_{\R_0}\bigg(e^{c\rho(t,z)}-1-c\rho(t,z)\bigg)\nu(dz)\bigg),\\
& \fs(t,x)=xc\sigma(t),\\
&\fc(t,x,z)=x\bigg(e^{c\rho(t,z)}-1\bigg).
\end{align*}
By Assumption \ref{X-assumptions}, we have for some $K>0$,
\begin{align*}
|\fb(t,x)|&\leq \bigg[cK_b+\frac{c^2K_\sigma^2}{2}+\int_\R\bigg(e^{c\kappa_{\rho}(z)}-1-c\kappa_{\rho}(z)\bigg)\nu(dz)+cL_b|X^t|_\infty\bigg]|x|\\
 &  \le K(|X^t |_\infty+1)|x| \end{align*}
so that
\begin{align*}
\fX_t \le  & e^{cx}+\int_0^t K ( \fX_s| \log(\fX^s)|_\infty+1) ds+ \int_0^t \fs(s,\fX_{s})dW_s\\
&+\int_{{]0,t]}\times{\R_0}}\fc(s,\fX_{s-},z)\tilde{\nN}(ds,dz),
\end{align*}
and further, estimating $\fX$ by its supremum  and using that $x (|\log x|+1)\le 2(x+2)\log(x+2)$,
\begin{align*}
\fX_t +2  \le &2+ e^{cx} +\int_0^t 2K ( |\fX^s+2|_\infty)\log(|\fX^s+2|_\infty)  ds\\ 
&+ \int_0^t \fs(s,\fX_{s})dW_s
+\int_{{]0,t]}\times{\R_0}}\fc(s,\fX_{s-},z)\tilde{\nN}(ds,dz).
\end{align*}
 By a  stochastic Bihari-LaSalle inequality (Gronwall type Theorem \ref{Sarah2}) we have, setting
$\eta(x)=x| \log x|, c_0=1, A_t =2Kt  $  and $H_t =2+   e^{cx}$, $M_t = \int_0^t \fs(s,\fX_{s})dW_s+\int_{ ]0,t] \times \R_0}\fc(s,\fX_{s-},z)\tilde{\nN}(ds,dz)$
for any $p \in (0,1)$ that 
\begin{align*}
\EE \Big|G^{-1}\Big(G(|\fX+2|_\infty)-\frac{2K}{1-p}T\Big)\Big |^p\leq \frac{(2 +e^{cx})^p }{1-p}<\infty,
\end{align*}
where $G(x)=\int_r^x\frac{ds}{s\log(s)}=\log\log(x)-\log\log(r)$ for some $r>1$, and thus
$G^{-1}(x)=e^{\log(r)e^x}$  hence,
 $$G^{-1}\big(G(|\fX+2|_\infty)-(1-p)^{-1}T\big)=|\fX+2|_\infty^{e^{-(1-p)^{-1}T}}.$$
 Therefore, 
$$\EE e^{cpe^{-(1-p)^{-1}  T}\sup_{0\leq t\leq T }X_t} = \EE |\fX^{pe^{-(1-p)^{-1}T}}|_\infty \leq \EE |\fX +2|_\infty^{pe^{-(1-p)^{-1}T}}<\infty.$$
Since $c$ can be chosen arbitrary, the assertion follows.
\end{proof}

\section{Malliavin derivatives on the L\'evy-It\^o  space} \label{Mall}
We use  $(\Omega,\mathcal{F},\mathbb{P})$, $W,$ $\tilde \nN$ etc as introduced in Section \ref{setting}.  
Setting
\[\mu(dx):=\delta_0(dx)+\nu(dx) \quad \text{and} \quad  \m(dt,dx) :=(\lambda\otimes\mu) (dt,dx), \]
we define an independently scattered random measure $\mM$ by 
\equal  \label{measureM}
   \mM(dt,dx):= dW_t\delta_0(dx) +\tilde \nN(dt,dx)
\tionl
 on sets $B \!\in \!\mathcal{B}([0,T]\times\R)$ with  $\m(B) < \infty$.
It holds $\E \mM(B)^2 = \m(B).$
According to \cite{ito} there exists for any 
$\xi \in L^2(\Om,\ftn,\mathbb{P})$ a unique chaos expansion 
\equal \label{chaos}
\xi=\sum_{n=0}^\infty I_n(\tilde f_n),
\tionl
where  $f_n \in  L^n_2:=L^2(([0,T]\times\R)^n,\bB(([0,T]\times\R)^n), \m^{\otimes n}),$ and  the function  $\tilde f_n((t_1,x_1),...,(t_n,x_n))$ is the symmetrization of 
$f_n((t_1,x_1),...,(t_n,x_n))$ w.r.t.~the $n$ pairs of variables. The multiple integrals $I_n$ are build with the random measure  
$\mM$ from \eqref{measureM}.  For their definition and properties see  \cite{ito} or \cite[Section 2]{GeissStein24}. Let $\DD$ be the space  of all  random variables $\xi \in L^2(\Om,\ftn,\mathbb{P})$  such that
\equa
   \|\xi\|^2_{\DD}:= \sum_{n=0}^\infty (n+1)!\left\|\tilde f_n\right\|_{L^n_2}^2<\infty.
\tion
For $\xi \in \DD$   the Malliavin derivative is defined by
\begin{equation*}
D_{t,x}\xi:=\sum_{n=1}^\infty nI_{n-1}\left(\tilde f_n\left((t,x),\ \cdot\ \right)\right),
\end{equation*}
for $\mathbb{P}\otimes\m$-a.a. $(\omega,t,x)\in\Omega\times{[0,T]}\times\R$. It holds $D \xi  \in  L^2(\mathbb{P}\otimes\m )$.  \bigskip

Defining
\equa
\mathbb{D}_{1,2}^0&:= &\bigg \{\xi=\sum_{n=0}^\infty I_n(\tilde f_n) \in L^2(\Om,\ftn,\mathbb{P})\colon  f_n \in  L_2^n, n\in \N,\\
&&  \quad \quad \quad \quad \quad \quad\quad \quad \quad\sum_{n=1}^\infty  (n+1)! \int_0^T  \|\tilde f_n((t,0),\cdot) \|_{L^{n-1}_2}^2 dt < \infty  \bigg \}
\tion
and
\equa
\mathbb{D}_{1,2}^{\R_0}&:= & \bigg \{\xi=\sum_{n=0}^\infty I_n(\tilde f_n) \in L^2(\Om,\ftn,\mathbb{P})\colon f_n \in L_2^n, n\in \N, \\
   &&  \quad \quad \quad  \sum_{n=1}^\infty  (n+1)! \int_{[0,T]\times \R_0}  \|\tilde f_n((t,x),\cdot)
   \|_{\mathrm{L}^{n-1}_2}^2 \m( dt,dx) < \infty  \bigg \}.
\tion
we get that $\DD =  \mathbb{D}_{1,2}^0\cap \mathbb{D}_{1,2}^{\R_0}.$ \bigskip

 The next result was proved in  \cite{Nualart} in the Wiener space setting. We  use similar ideas to show that it holds also  
for $\mathbb{D}_{1,2}^0$ and $\mathbb{D}_{1,2}^{\R_0}.$

\begin{lemma}[{\cite[Lemma 1.2.3]{Nualart}}] \label{lemma123}
 Let
 \begin{enumerate}[(i)]
 \item  $(\xi_N)_{N=1}^\infty \subseteq \mathbb{D}_{1,2}^0$\,\,  with \,\, $ \sup_N \EE \int_0^T|  D_{t,0}  \xi_N| ^2 dt < \infty, \label{d0}  $ 
 \item   or $(\xi_N)_{N=1}^\infty \subseteq \mathbb{D}_{1,2}^{\R_0}$  \,\,  with \,\,  $\sup_N \EE   \int_{[0,T]\times \R_0}   |D_{t,x} \xi_N| ^2\m( dt,dx) < \infty. \label{dx}  $
  \end{enumerate}
If  $\xi_N \to \xi$ in $L^2(\Om,\ftn,\mathbb{P}),$   then  assumption \eqref{d0}  implies that $\xi \in \mathbb{D}_{1,2}^0$, and assumption  \eqref{dx}  implies that $\xi \in \mathbb{D}_{1,2}^{\R_0}.  $
\end{lemma}
\begin{proof}
Assume that $
\xi_N=\sum_{n=0}^\infty I_n(\tilde f^N_n).$  Then  $\E |\xi_N|^2 =   \sum_{n=0}^\infty n!\big\|\tilde f^N_n\big\|_{L^n_2}^2$  
and  $$D_{t,x} \xi_N = \sum_{n=1}^\infty nI_{n-1}\left(\tilde f^N_n\left((t,x),\ \cdot\ \right)\right). $$ 
To show  \eqref{d0} we  put $A=[0,T]\times \{0\}$. 
Hence   by the orthogonality of the $I_{n-1}\left(\tilde f^N_n\left((t,x),\ \cdot\ \right)\right)$,
\begin{align*}
\EE \int_0^T|  D_{t,0}  \xi_N| ^2 dt &=\E  \int_A  |D_{t,x}\xi_N|^2 \m( dt,dx)  \\
&= \E \!   \int_{[0,T]\times \R} \! \big |  \sum_{n=1}^\infty  nI_{n-1}\left(\tilde f^N_n\left((t,x),\ \cdot\ \right)\right)  \one_A(t,x) \big |^2  \! \m( dt,dx) \\
&= \sum_{n=1}^\infty n n! \int_0^T  \|\tilde f^N_n((t,0),\cdot) \|_{L^{n-1}_2}^2 dt.
\end{align*}
By assumption,  we have that for all $K, N\in \N$
$$   \sum_{n=1}^K  n n! \int_0^T  \|\tilde f^N_n((t,0),\cdot) \|_{L^{n-1}_2}^2 dt \le \EE \int_0^T| D_{t,0}  \xi_N| ^2 dt \le   C.  $$
On the other hand,  $\int_0^T  \|\tilde f^N_n((t,0),\cdot) \|_{L^{n-1}_2}^2 dt \to \int_0^T  \|\tilde f_n((t,0),\cdot) \|_{L^{n-1}_2}^2 dt$ since $\xi_N \to \xi$ in $L^2(\Om,\ftn,\mathbb{P})$.
Therefore, for all $K\in \N$ 
$$ \sum_{n=1}^K  n n! \int_0^T  \|\tilde f_n((t,0),\cdot) \|_{L^{n-1}_2}^2 dt  \le   C $$
which means $$\EE \int_0^T|  D_{t,0}  \xi| ^2 dt \le C.$$
By replacing  $A=[0,T]\times \{0\}$  by  $A= [0,T]\times \R_0$ one shows in the same way that   assumption \eqref{dx} implies $   \xi \in \mathbb{D}_{1,2}^{\R_0}. $
\end{proof}

\subsection{A criterion for \texorpdfstring{$\mathbb{D}_{1,2}^0$}{Brown}  }

If $W$ and $W'$ are independent Brownian motions  on   $(\Omega,\mathcal{F},\mathbb{P})$ we define a 'coupled' Brownian motion $W^\varphi$ by setting
$$W^\varphi := \sqrt{1-\varphi^2} W + \varphi W',$$
where $\varphi \in [0,1]$ is a fixed constant.
If $\xi \in L^2(\Om,\ftn,\mathbb{P})$ is given as in \eqref{chaos}
then $\xi^\varphi$ is built by the same kernels but with respect to the random measure 
$$ \mM^\varphi(dt,dx):= dW^\varphi_t\,\delta_0(dx) +\tilde \nN(dt,dx).$$
Using a functional representation of $\xi \in L^2(\Om,\ftn,\PP)$, i.e.~$\xi=\Xi(\mathcal{L})$, with measurable $\Xi\colon D[0,T]\to\R$, and the L\'evy - It\^o decomposition $ \mathcal{L}=\varsigma W+ \mathcal{L}-\varsigma W$, the 'coupled' random variable is then given by $\xi^\varphi=\Xi(\varsigma W^\varphi+\mathcal{L}-\varsigma W)$ (see \cite{Steinicke}).

In  \cite{GeissZhou} the following criteria for $ \mathbb{D}_{1,2}^0$ was shown  (the results are true also for a Hilbert space valued random variable):

\begin{thm}[ {\cite[Corollary 4.5]{GeissZhou}} ] \label{StefanXilin}
Let $\xi \in L^2(\Om,\ftn,\mathbb{P})$. Then 
$$\xi \in \mathbb{D}_{1,2}^0   \iff   \sup_{0<\varphi\le 1} \frac{\|\xi -\xi^\varphi\|_{L^2}}{\varphi} < \infty.$$
Moreover, if  $\xi \in \mathbb{D}_{1,2}^0 $ then one has
$$ \frac{1}{2}   \left (\E \int_0^T| D_{t,0}\xi  |^2 dt \right)^\frac{1}{2}     \le      \sup_{0<\varphi\le 1} \frac{\|\xi -\xi^\varphi\|_{L^2}}{\varphi}  \le  \left (\E \int_0^T| D_{t,0}\xi  |^2 dt \right )^\frac{1}{2} .    $$
\end{thm}

\subsection{A criterion for \texorpdfstring{$\mathbb{D}_{1,2}^{\R_0}$}{jump}  }
If $g: D{[0,T]} \to \R$
is $\mathcal{G}_T$-measurable (recall \eqref{filtrationG-t}) one can compute the Malliavin derivative in the jump direction i.e.~$D_{t,x}$ for $x \neq 0$    using the following result:

\begin{lem}[\cite{Steinicke},  {\cite[Lemma 3.2]{GeissStein16}} ] \label{functionallem}
If  $g: D{[0,T]} \to \R$
is $\mathcal{G}_T$ measurable   and  $g(\mathcal{L}) = g(( \mathcal{L}_t)_{t \in [0,T]}) \in  L^2(\Om,\ftn,\mathbb{P})$
then 
\equa
 g(\mathcal{L}) \in  \mathbb{D}_{1,2} ^{\R_0}  \iff  g(\mathcal{L}+x\one_{[t,T]})-g( \mathcal{L}) \in L^2(\PP \otimes\m),
\tion
and it holds then  for  $\PP \otimes\m$-a.e.   $(t,x) \in [0,T]\times \R_0$ that 
\begin{equation}\label{xieq}
D_{t,x}  g(\mathcal{L}) =  g( \mathcal{L}+x\one_{[t,T]})-g(\mathcal{L}) .
\end{equation}
\end{lem}

\begin{cor}
    \label{malliavin-derivative-explicite}
     Let $X$ be a strong solution  to \eqref{forward} and assume that $X_T \in   \mathbb{D}_{1,2} ^{\R_0}$.    If $F: \R\to \R$ is Borel measurable and $F(X_T + D_{t,x}X_T) -F(X_T) \in L^2(\PP \otimes\m)$, then   for $  \PP \otimes\m $-a.e.~$(t,x) \in [0,T]\times \R_0 $
    $$ D_{t,x} F(X_T) =  F(X_T + D_{t,x}X_T) -F(X_T). $$
  \end{cor}
\begin{proof}
    If $X$ is a strong solution we can represent  $X_T$ as $X_T = G(( \mathcal{L}_t)_{t\in [0,T]})$  a.s., where  $G:D{[0,T]} \to \R $ is  $\mathcal{G}_T$ measurable   (see \cite[Corollary 5.41]{Menaldi}), and get by \eqref{xieq} in Lemma \ref{functionallem}
    \begin{align*}
   & D_{t,x} F(X_T) =  D_{t,x} F\Big(G ( \mathcal{L})\Big)= F(G( \mathcal{L}+x\one_{[t,T]}))-F(G( \mathcal{L}))\\
   &=F( G(\mathcal{L})+D_{t,x}G( \mathcal{L}) )-F(G(\mathcal{L}))  = F(X_T +D_{t,x} X_T) -F(X_T).
    \end{align*}
   
\end{proof}

We recall another result we will use:

\begin{lem}[{\cite[Lemma 3.3]{GeissStein16}} ]
\label{shiftthm}
Let $\Lambda\in \mathcal{G}_T$ be a set with  $\mathbb{P}\left(\left\{ \mathcal{L}\in \Lambda\right\}\right)=0$. Then
$$\mathbb{P} \otimes\m\left(\left\{(\omega,s,x)\in \Omega\times{[0,T]}\times\R_0:\mathcal{L}(\omega)+x\one_{[s,T]}\in \Lambda\right\}\right)=0.$$ 
\end{lem}

\subsection{Malliavin differentiability  and bounds for the SDE}

\begin{prop}  \label{bounds-for-DX}   Let the  Assumption \ref{X-assumptions}  hold. Then it holds for the  solution $X$ to \eqref{forward} that $\, X_t \in \mathbb{D}_{1,2} \,$ for all $t \in [0,T].$ 
Moreover,
\begin{enumerate}[(1)]
 \item for $\nu$-a.a.  $x\in \R_0 $  and  for a.a. $s\leq t\leq T$ we have
\equa  D_{s,x} X_t &=&  \rho(s,x) + \int_s^t  \big(b(r,X^r+ D_{s,x} X^r ) - b(r,X^r )\big)dr,     
  \tion
   \item and for $x = 0$ and a.a.  $s \leq t \leq T$,
\equa  D_{s,0} X_t &=&  \sigma(s) + \int_s^t  D_{s,0} \,  b(r,X^r)  dr.   
 \tion
 Moreover, we have for  a.a. $s \leq t \leq T$
\begin{equation}
\label{upperbound-derivative-X}
    |D_{s,x} X_t|   \le  e^{L_b (t-s)} (\kappa_\rho(x) \one_{x \neq 0}  +   K_{\sigma} \one_{x = 0}).
\end{equation}
\end{enumerate}
\end{prop}

\begin{proof}
$\boxed{ \, X_t \in   \mathbb{D}_{1,2}^0}:$     \\
By  Theorem \ref{StefanXilin} we need to estimate  $\|X_t -X^\varphi_t\|_{L^2},$
where 
\begin{align*} 
X^\varphi_t &= \,  x + \int_0^t b(s,X^{s,\varphi}) ds +  \int_0^t \sigma(s) d(\sqrt{1-\varphi^2} W_s + \varphi W'_s)  \\
& \quad + \int_{{]0,t]}\times{\R_0}}  \rho(s,v)  \tilde \nN(ds,dv).
  \end{align*}
For later use we compute here $\| |X^t -X^{t,\varphi}|_\infty \|_{L^p}$ for $p\ge 2.$ By the Burkholder-Davis-Gundy inequality,
\begin{align*}   
\||X^t -X^{t,\varphi} |_\infty\|_{L^p} &\le  \int_0^t \|b(s,X^s)  -   b(s,X^{s,\varphi})\|_{L^p} ds  \notag\\
& \quad + C_p (1- \sqrt{1-\varphi^2} +\varphi ) \left ( \int_0^t \sigma(s)^2 ds \right)^\frac{1}{2}.
  \end{align*} 
Using  $\|b(s,X^s)  -   b(s,X^{s,\varphi})\|_{L^p}\le L_b  \||X^s  -   X^{s, \varphi}|_\infty\|_{L^p}$, Gronwall's inequality and $1- \sqrt{1-\varphi^2} \le \varphi$
imply 
\begin{align} \label{the-p-estimate}
\||X^t -X^{t,\varphi} |_\infty\|_{L^p} &\le c\varphi
 \end{align}
which especially means $\, X_t \in   \mathbb{D}_{1,2}^0$  if $p=2.$ \smallskip
 $\boxed{\, X_t \in \mathbb{D}_{1,2}^{\R_0}:}$  \smallskip\\
 If $X$ is a strong solution  \cite[Corollary 5.41]{Menaldi} gives us the representation 
     $X_t=G(t,\mathcal{L}^{ t })$ for all $t \in [0,T],$   a.s.~where $G(t, \cdot) :D{[0,T]} \to \R $ is  $\mathcal{G}_t$ measurable.
 Then by \eqref{xieq}  we need 
 $ (\omega,r,v) \mapsto ( G(t, \mathcal{L}^t + v\one_{[r,t]} )- G(t,  \mathcal{L}^t ) )  \in L^2(\PP \otimes\m).$ 
   Abbreviating $ \mathcal{X}^{r,v}_t := G(t,  \mathcal{L}^t + v\one_{[r,t]} )$    we conclude  from \eqref{forward} that
\begin{align*}
 \mathcal{X}^{r,v}_t   &= x+    \int_0^t b\left (s,(\mathcal{X}^{r,v})^s \right )  ds   +  \int_0^t \sigma(s) dW_s \\
& \quad + \rho(r,v)\one_{{]0,t]}\times{\R_0}}(r,v)+ \int_{{]0,t]}\times{\R_0}}  \rho(s,z)  \tilde \nN(ds,dz) ,
\end{align*}   
where we used \eqref{xieq}  again and the fact  that $D_{r,v}  \int_{{]0,t]}\times{\R_0}}  \rho(s,z)  \tilde \nN(ds,dz) =  \rho(r,v)$ since the integral belongs to the first chaos.  Approximating $ \mathcal{X}^{r,v}_t$  as well as  $X_t$   via  Picard iteration starting with  $\mathcal{X}^{r,v,0}_t := 0$ and $X^0_t=0$ one derives then 
 \begin{align*}
 \mathcal{X}^{r,v,n+1}_t   - X_t^{n+1}   =&    \int_0^t \Big( b\left (s,(\mathcal{X}^{r,v,n})^s \right )  -  b \left (s,(X^{n})^s\right ) \Big)   ds \\
 &+ \rho(r,v)\one_{{]0,t]}\times{\R_0}}(r,v),
\end{align*}    
and by the Burkholder-Davis-Gundy inequality,
\begin{align*}
& \| (\omega,r,v) \mapsto  |(\mathcal{X}^{r,v,n+1})^t   - (X^{n+1})^t|_\infty ) \|_{ L^2(\PP \otimes\m)} \\
& \le    \int_0^t \| (\omega,r,v) \mapsto ( b\left (s,(\mathcal{X}^{r,v,n})^s \right )  -  b \left (s,(X^{n})^s\right )  )  \|_{ L^2(\PP \otimes\m)}  ds +  C_2 \kappa_{\rho,2} \\
&\le L_b \int_0^t \| (\omega,r,v) \mapsto   |(\mathcal{X}^{r,v,n})^s -  (X^{n})^s |_\infty \|_{ L^2(\PP \otimes\m)}  ds +   C_2 \kappa_{\rho,2}.
\end{align*}    
Since    $\| (\omega,r,v) \mapsto  |(\mathcal{X}^{r,v,1})^t   - (X^{1})^t|_\infty ) \|_{ L^2(\PP \otimes\m)}   \le  C_p\, \kappa_{\rho,2}<\infty$ one can derive by Gronwall's inequality from this relation that  $$\| (\omega,r,v) \mapsto  (G(t, \mathcal{L}^t + v\one_{[r,t]} )   - G(t, \mathcal{L}^t )) \|_{ L^2(\PP \otimes\m)}< \infty.$$ \smallskip
For $x\neq 0$ we get  \eqref{upperbound-derivative-X} by Assumption \ref{X-assumptions} 
\begin{align*} 
|D_{s,x} X_t|  & \le   |D_{s,x} X^t|_\infty \le   \kappa_\rho(x)  + \int_s^t L_b  |D_{s,x} X^r |_\infty dr,     
\end{align*} 
from which the assertion follows by Gronwall's inequality.
\end{proof}


\subsection{The Malliavin derivative of a functional}

We prove a representation for  the Malliavin derivative $D_{t,0}$ of the generator of the BSDE. In our setting the generator is a  function depending pathwise on  the process $X$,  and on random variables. The proof  uses a characterisation of $ \mathbb{D}_{1,2}^0$ done by Sugita \cite{Sugita}, it is postponed to the Appendix \ref{proof-of-chain-rule}.

\begin{lem}  \label{chain-rule-lemma} Suppose Assumptions  \ref{X-assumptions} and let  $X$ be the solution to \eqref{forward}. 
\begin{enumerate}[(1)]
\item  \label{X-estimate}   For a.e. $s \in [0,t]$ it holds
$    |D_{s,0} X_t|   \le  e^{L_b t}    K_{\sigma}. $
\item  \label{bold-f-estimate} Assume that 
$\f: [0,T]\times D([0,T]) \times \R^d \to \R $ is measurable 
satisfying 
\begin{align*} 
 |\f(t,0,0)|& \le   k_f(t),\\
 |\f(t,\bx, y)-f(t,\bx', y) |   
 &\le L_\bx \,\,  (c+ \frac{\beta}{2}(|\bx|^ {\tt r}_\infty + |\bx'|^ {\tt r}_\infty))   |\bx-\bx'|_\infty, \\
 |\f(t,\bx, y)- \f(t,\bx, y') | 
&\le  \sum_{k=1}^d  L_{y_k} \,
\, |y_k- y_k'|,
\end{align*} 

for some $L_\bx, L_{y_1},...,L_{y_d} \ge 0$ and     $k_f   \in L^1([0,T]).$  If $ Y_1, ...,Y_d \in \mathbb{D}_{1,2}^{0},$  
then 
there exist  measurable    $G_1, ... ,G_d : \Omega \times [0,T] \to \R$ which are bounded: $|G_k| \le  L_{y_k},$  and  such that  it holds a.s. and for a.a. $s\in [0,T]$
\begin{align}  \label{chain-rule}
  D_{s,0}  \f  \left( t, X^t, (Y_1, ..., Y_d)   \right)   &=  ( D_{s,0}  \f ( t, X^t,y) ) |_{y=(Y_1, ... ,Y_d )} \notag\\
  &\quad  +  G_1(t) \, D_{s,0} Y_1  +... + G_d(t) \, D_{s,0} Y_d  
\end{align}
and
\begin{align} \label{Df-estimate}  
\|  D_{\cdot,0}  \f ( t, X^t,y) ) \|_{L^\infty[0,T]}  \le   L_\bx  K_\sigma  
\big ( c+  \beta |X^t|^ {\tt r}_\infty \big) e^{L_b T}.
\end{align}
for all  $y \in \R^d$.
\end{enumerate}

 \end{lem}

\section{Malliavin differentiable, bounded solutions of the  BSDE under Lip\-schitz conditions}\label{sec: bounds}

\subsection{Truncation and Malliavin differentiability}

\rm We show now that the truncated version of the BSDE \eqref{fbsde}  is Malliavin differentiable.

For $M >0$ let  { $b_M :\R \to [-M,M]$} be a smooth monotone function such that  $0\le b'_M(x) \le 1$ and
\equa
b_M(x) := \left \{ \begin{array}{cl} M, &\,\, x>M+1,\\
x,  & \,\, |x| \le  M -1, \\
-M, &\,\, x< -M-1.
\end{array} \right .
\tion
Note that $|b_M(x)|\le |x|\wedge M$. { By a slight abuse of notation, we also define, for all $\bx \in D[0,T]$, $b_M(\bx)$ as the function of $D[0,T]$ given by $t \mapsto b_M(\bx_t)$.} \\
We set
\begin{align} \label{data-cut-off}
  g^M(\bx)&:= g(b_M(\bx))  \notag \\
    f^M(s,  \bx ,y, z, u )&:=  f\Big( s, b_M(\bx) ,y, b_M(z), b_M(u) \Big).
     \end{align} 
     
Theorem \ref{StefanXilin} implies  Malliavin differentiability of the truncated terminal condition:
\begin{cor}\label{cor1} Let the  Assumptions \ref{X-assumptions} and \ref{Ypath-assumptions}   hold. Then it holds for $\varphi \in [0,1]$ that
$$\|g^M(X)- g^M(X^\varphi) \|_{L^2}  \le c L_{g^M}  \, \varphi,$$
where  $L_{g^M} = (c +  \alpha |M|^{\tt r}).$
\end{cor} 
\begin{proof}
By  the Burkholder-Davis-Gundy  inequality we get similarly to the proof of Proposition \ref{bounds-for-DX}   that $\| |X-X^\varphi |_\infty\|_{L^2} \le  c\varphi.$
Then by Assumption \ref{Ypath-assumptions} \eqref{the-g} , 
\begin{align*}
 \|g^M(X)- g^M(X^\varphi) \|_{L^2}  &\le (c +  \alpha |M|^{\tt r})\| |b_M(X)-b_M(X^\varphi) |_\infty\|_{L^2}
    \le c L_{g^M} \, \varphi.
\end{align*} 
\end{proof}

\begin{prop} \label{Malliavin-diff-trunc-BSDE}   Let the  Assumptions \ref{X-assumptions} and \ref{Ypath-assumptions} hold. There  exists a unique  solution $(Y^M,Z^M,U^M)  \in \mathcal{S}^2 \times L^2 (W)\times L^2(\tilde \nN)$  to 
\begin{align} \label{BSDE-truncated}
 Y^M_t= \,&g^M(X)+\int_t^T  f^M \left( s, X^s,\Theta^M_s \right)ds  \notag \\
& -     \int_t^T Z^M_{s}   dW_s -\int_{{]t,T]}\times{\R_0}}U^M_{s}(v) \tilde \nN(ds,dv)
\end{align}
where $\Theta^M_s= (Y^M_s, Z^M_s,  H^M_s) , $     and 
\begin{align} \label{the-H}
 H^M_s =  \int_{\R_0}  h(s, b_M ( U^M_s(v))) \kappa (v)  \nu(dv)
 \end{align}
 Moreover, the solution processes  $(Y^M,Z^M, U^M)$  are Malliavin differentiable, i.e.
$$
Y^M, Z^M   \in L^2([0,T];\DD), \quad U^M\in L^2([0,T]\times\R_0;\DD),
$$
and  for $t \le u\le T$ we have that
\begin{align} \label{Diff-BSDE}
D_{t,x} Y^M_u= \,& D_{t,x}g^M(X)+\int_u^TD_{t,x}  f^M \left( s, X^s, \Theta^M_s  \right)ds  \notag \\
& -     \int_u^TD_{t,x} Z^M_{s}   dW_s -\int_{{]u,T]}\times{\R_0}}  D_{t,x} U^M_{s}(v) \tilde \nN(ds,dv).
\end{align}
\end{prop}
\begin{proof}  The existence and uniqueness of the solution is clear.
Since we are in the Lipschitz case with bounded terminal condition,
 \equa
Y^M, Z^M   \in L^2([0,T]; \mathbb{D}_{1,2}^{\R_0}), \quad U^M\in L^2([0,T]\times\R_0; \mathbb{D}_{1,2}^{\R_0})
\tion
is shown in  \cite[Theorem A1]{GeissStein20}. However, for  $\mathbb{D}_{1,2}^0$ the additional regularity condition  \cite[condition $(A_f)$ e)]{GeissStein20}  is used  which we do not  want to  assume here.
Instead we will apply repeatedly Theorem  \ref{StefanXilin}.  We consider
\begin{align*} 
 Y^{M,\varphi}_t= \,&g^M((X)^\varphi)+\int_t^T  f^M \left( s, (X^s)^\varphi,  \Theta^{M,\varphi}_s  \right)ds  \notag \\
& -     \int_t^T Z^{M,\varphi}_{s}   d (\sqrt{1-\varphi^2} W_s + \varphi W'_s)  -\int_{{]t,T]}\times{\R_0}}U^{M,\varphi}_{s}(v) \tilde \nN(ds,dv).
\end{align*}

We introduce the notation   $( \Delta \xi, \Delta Y_t, \Delta Z_t,  \Delta U_t ):=( g^M(X)-g^M((X)^\varphi),  Y^M_t -  Y^{M,\varphi}_t,  Z^M_t -  Z^{M,\varphi}_t,U^M_t -  U^{M,\varphi}_t).$
From \cite[Theorem 6.3.]{GeissYlinen} and \cite[Proposition 2.2]{bbp} one concludes   an a priori  estimate:  there is a  $C>0$   such that for all $t\in [0,T]$,
  \begin{align}  \label{apriori-levy-Ito}
  &\EE\sup_{s\in[t,T]}|\Delta Y_s|^2+  (1-\sqrt{1-\varphi^2})  \EE\int_t^T|Z_s^{ M}|^2ds   + \EE\int_t^T|\Delta Z_s|^2ds  \notag \\
  & \quad+\EE\int_t^T \|\Delta U_s\|_{L^2(\nu)}^2ds \notag\\
  &\leq C\EE\bigg[|\Delta \xi|^2+\bigg(\int_t^T |f^M \left( s, X^s,  \Theta^M_s \right)-  f^M \left( s, (X^s)^\varphi,\Theta^M_s  \right)|ds\bigg)^{2}\bigg].\nonumber  \\
  &\leq C( L_{g^M}^2  +  (T-t)^2  L_{ f^M}^2) \,\, c^2\, \varphi^2
  \end{align}
  with  $\Theta^M_s:= (Y^M_s, Z^M_s,  H^M_s ).$
From this we immediately derive that  
$$ Y^M_t  , \,\,\, \int_t^T Z^M_{s}   dW_s , \,\,\, \int_{{]t,T]}\times{\R_0}}U^M_{s}(v) \tilde \nN(ds,dv) \in \mathbb{D}_{1,2}^0.$$
We  also have $ Z^M   \in L_2([0,T]; \mathbb{D}_{1,2}^{0})$ since 
$$    \int_t^T  \|Z^M_s -  Z^{M,\varphi}_s\|^2_{L^2} ds  =   \int_t^T \EE |\Delta Z_s|^2ds $$
so that we again can use  \eqref{apriori-levy-Ito}   and Theorem \ref{StefanXilin}. By the same argument we also have
$  U^M\in L^2([0,T]\times\R_0;\mathbb{D}_{1,2}^{0})$.
From the last inequality in \eqref{apriori-levy-Ito} one also  concludes that $ \int_t^T f^M \left( s, X^s,  \Theta^M_s \right)ds \in \mathbb{D}_{1,2}^0.$   
By   \cite[Lemma 3.2, Lemma 3.3]{DelongImkeller10}  we can interchange $D_{t,x}$ with  integrals so that
\begin{align*} 
D_{t,x} Y^M_u= \,& D_{t,x}g^M(X)+\int_u^TD_{t,x}  f^M \left( s, X^s,  \Theta^M_s  \right)ds  \notag \\
& -     \int_u^TD_{t,x} Z^M_{s}   dW_s -\int_{{]u,T]}\times{\R_0}}  D_{t,x} U^M_{s}(v) \tilde \nN(ds,dv).
\end{align*}
\end{proof}


\subsection{Bounds and representations for  \texorpdfstring{$Z^M$}{Z-M} and \texorpdfstring{$U^M$}{U-M}}

Based on a comparison theorem, we get a first preliminary bound which depends on the truncation level $M$. 

\begin{prop} \label{first-estimates-depending-on-M-pi-eps}  Suppose Assumptions  \ref{X-assumptions} and \ref{Ypath-assumptions} hold.  
 For fixed $M>0$ there is a constant $a_0^M>0$   such that   $\PP$ -a.s.
  we have that for $\lambda$-a.e.~$t \in [0,T]$
   $$ |Z^M_t|  \le a_0^M$$
   and  for $dt\otimes \nu $ a.e. $(t,x) \in [0,T]\times \R_0$
  $$ |U^M_t(x)| \le \kappa_{\rho}(x)a_0^M.$$
  \end{prop}
\begin{proof}  

 By Proposition \ref{Malliavin-diff-trunc-BSDE} we know that $(Y^M,Z^M,U^M)$ is Malliavin differentiable, and we have  \eqref{Diff-BSDE}.
Since  $ (\lim_{u \downarrow t} D_{t,0} Y^M_u)_{t\in [0,T]}$ and  $( \lim_{u \downarrow t} D_{t,x} Y^M_u)_{t\in [0,T]}$ for $x\neq 0$ 
have  c\`adl\`ag adapted versions, we get by \cite[Theorem 3.7 and Lemma 3.5]{Chung} that 
    the predictable projections    coincide with   the processes itself up to sets of measure $dt \otimes \PP$ zero, which implies 
 \begin{align} \label{ZandUlimDY}
 Z^M_t=   \lim_{u \downarrow t} D_{t,0} Y^M_u  \,\, a.s. \quad   \text{and} \quad    U^M_t(x) =   \lim_{u \downarrow t} D_{t,x} Y^M_u\,\,  a.s.
 \end{align}
 for $\nu$ a.a. $ x \in \R_0.$   \\
 By assumption we have that  
$$  |g^M(\bx)- g^M(\bx')| \le ( c +\alpha M^\tr) \, \dJ(\bx,\bx'). $$
Then Lemma  \ref{chain-rule-lemma}  (seeing $g^M$ as a special case of $\f$)  implies
\begin{align*}
 |D_{t,x}  g^M(X)|& \le( c +\alpha M^\tr)   |D_{t,x } X|_\infty \le ( c +\alpha M^\tr) e^{L_b T} (\kappa_\rho(x) \one_{x \neq 0}  +   K_{\sigma} \one_{x = 0}) \\
 &=: \xi_{M,x} 
  \end{align*} 
In the same way we get   setting $H(s,{\bf u}) :=  \int_{\R_0}  h(s, b_M ( {\bf u}(v))) \kappa (v)  \nu(dv)$ from
\begin{align*}
&  |  f^M(s,\bx^s, y,z,   H(s,{\bf u}) )  - f^M(s,(\bx')^s, y',z',   H(s,{\bf u}')  ) |  \\
  & \le    (c +\beta M^\tr) \, \dJ(\bx,\bx') + L_{f ,\ry} |y-y'| +  (c +\gamma  M^\ell)|z-z'|    \\
  &\quad   + (c+ \gamma M^{m_1}) \left |   \int_{\R_0} ( h(s, b_M ( {\bf u}(v)))-h(s, b_M ( {\bf u}'(v)))) \kappa (v)  \nu(dv)  \right | \\
  & \le    (c +\beta M^\tr) \, \dJ(\bx,\bx') + L_{f ,\ry} |y-y'| +  (c +\gamma  M^\ell)|z-z'|    \\
  &\quad   + (c+ \gamma M^{m_1})   (c +\gamma  M^{m_2})  \|{\bf u}-{\bf u}'\|_{L^2(\nu)}  \, \| \kappa (v)\|_{L^2(\nu)}
   \end{align*} 
 that    
    $$|D_{t,x} f^M(s,X^s, y,z,  H(s,{\bf u}) )|  \le    (c +\beta M^\tr) e^{L_b T} (\kappa_\rho(x) \one_{x \neq 0}  +   K_{\sigma} \one_{x = 0}) =: c_{M,x} $$
 and for $x \in \R$ we have
 \begin{align}
 |D_{t,x} f^M(s,X^s, \Theta^M)|  &\le     c_{M, x} +   L_{f ,\ry} | D_{t,x} Y^M_s |  +   (c +\gamma  M^\ell)| D_{t,x} Z^M_s | \notag\\
 & \quad +  (c+ \gamma M^{m_1})   (c +\gamma  M^{m_2})     \int_{\R_0} | D_{t,x} U^M_s(v)|  \kappa (v)  \nu(dv) \notag \\
 &=: f_{M,x}^+ (s,D_{t,x} Y^M_s, D_{t,x} Z^M_s, D_{t,x} U^M_s) \label{eq1}
\end{align}

In \cite[Theorem 3.4]{GeissStein20} using a comparison result it is shown  that   for the BSDEs   with data   $(\pm \xi_{M,x}, \pm f_{M,x}^+)$
we have that  for the  corresponding solution  processes  $\pm \overline\Y^{t,0}$ for  $(\pm \xi_{M,0}, \pm f_{M,0}^+)$  and    $(\pm \xi_{M,x}, \pm f_{M,x}^+)$ ($x\neq 0)$ for   $\pm \overline\Y^{t,x}$   that  for $0\le t\le u\le T$ 
$$  -  \overline \Y^{t,0}_u \le D_{t,0} Y^M_u \le \overline\Y^{t,0}_u$$
and 
$$  -  \overline \Y^{t,x}_u \le D_{t,x} Y^M_u \le \overline\Y^{t,x}_u.$$
Moreover,  since  $\xi_{M,0}$ is just a constant, and from the structure of $f_{M,0}^+$ given in  \eqref{eq1} one concludes by \cite[Theorem 2.3]{GeissStein20}  that there is a  unique solution to

 \begin{align*}
\overline\Y^{t,0}_u &= \xi_{M,0}   +\int_u^T f_{M,0}^+ (s, \overline\Y^{t,0}_s ,  \overline\Z^{t,0}_s, \overline \U^{t,0}_s)ds  \\ 
& \quad -  \int_u^T\overline\Z^{t,0}_{s}   dW_s -\int_{{]u,T]}\times{\R_0}}  \overline \U^{t,0}_{s}(v) \tilde \nN(ds,dv).
\end{align*}
which is  given by $(\overline \Y^{t,0}, 0,0).$  Hence 
\begin{align*}
Z^M_t=   \lim_{u \downarrow t} D_{t,0} Y^M_u  \le   \overline\Y^{t,0}_t  = \xi_{M,0} \, e^{L_{f ,\ry}(T-t)} +\int_t^T   c_{M,0} \, e^{L_{f ,\ry}(s-t)}  ds 
\end{align*}
with $\xi_{M,0} =  ( c +\alpha M^\tr) e^{L_b T}    K_{\sigma}$ and   $ c_{M,0}= (c +\beta M^\tr) e^{L_b T} K_{\sigma}.$
Similarly, 
 \begin{align*}
\overline\Y^{t,x}_u &= \xi_{M,x}   +\int_t^T f_{M,x}^+ (s, \overline\Y^{t,x}_s ,  \overline\Z^{t,x}_s, \overline \U^{t,x}_s)ds  \\
&\quad-     \int_t^T\overline\Z^{t,x}_{s}   dW_s -\int_{{]t,T]}\times{\R_0}}  \overline \U^{t,x}_{s}(v) \tilde \nN(ds,dv).
\end{align*}
implies \begin{align*}
U^M_t(x)=   \lim_{u \downarrow t} D_{t,x} Y^M_u  \le   \overline\Y^{t,x}_t  = \xi_{M,x} e^{L_{f ,\ry}(T-t)} +\int_t^T   c_{M,x} e^{L_{f ,\ry}(s-t)}  ds 
\end{align*}
with $\xi_{M,x} =  ( c +\alpha M^\tr) e^{L_b T}   \kappa_\rho(x)$ and   $ c_{M,x}= (c +\beta M^\tr) e^{L_b T} \kappa_\rho(x),$
so that we can choose
 $$a_0^M   :=  ( \xi_{M,0} \, e^{L_{f ,\ry}T} +T \,  c_{M,0}  \, e^{L_{f ,\ry}T} ) (1+   K_{\sigma}  ). $$

 \end{proof}

We continue with  the main result of this section which gives some uniform bounds that will be crucial to obtain our uniqueness and existence result for the BSDE \eqref{fbsde}.
    
\begin{prop} \label{bounds-independend-from-M}  Suppose Assumptions  \ref{X-assumptions} and \ref{Ypath-assumptions} hold.   Then there exist $a_\infty >0$ and $b_\infty>0$ (not depending on $ M$)   such that 
    \begin{align} \label{Z-and-U-pi-bound}
    |Z^M_t| & \le a_\infty +  b_\infty |X^t|_\infty^{\tr}   \notag \\
    |U^M_t(x)| &\le \kappa_{\rho}(x)(a_\infty+  b_\infty |X^t|_\infty^{\tr} ).
    \end{align}

    As a consequence we also get that there exists a $c_\infty >0$ (not depending on $M$) such that 
    \begin{align} \label{Y-pi-bound}
     |Y^M_t|  \le c_\infty(1 +   |X^t|_\infty^{\tr+1}).
     \end{align}
\end{prop}
\bigskip
The  proof uses  several technical lemmata  to which we turn next.
We  derive   representations for $Z^M$ and  $U^M$ from  \eqref{Diff-BSDE}   and  \eqref{ZandUlimDY}  writing 
  \begin{align} \label{Z-as-DY}    Z^M_t  \one_{x = 0} +  U^M_t(x) \one_{x \neq 0}   &=  D_{t,x} g^M(X^T) + \int_t^T  D_{t,x} f^M (s, X^s, \Theta^M_s )ds \notag \\
  & -     \int_t^T D_{t,x}   Z^M_s   dW_s  -\int_{{]t,T]}\times{\R_0}} D_{t,x}    U^M_s(v) \tilde \nN(ds,dv).  
  \end{align}
 where $\Theta^M_s= (Y^M_s, Z^M_s,  H^M_s) , $     and $ H^M_s =  \int_{\R_0}  h(s, b_M ( U^M_s(v))) \kappa (v)  \nu(dv).$

By the chain rule \eqref{chain-rule} denoting 
\begin{align}  \label{fhat-zero-terms}
\wf^0(s,  D_{t,0} X^s)&:= D_{t,0} f^M(s,X^s ,y,z,u)|_{(y,z,u)=(Y^M_s,Z^M_s,H^M_s)}  \notag \\
(\wf_{\ry}^0,  \wf_{\rz}^0,  \wf_{\ru}^0)&:=(G_1, G_2, G_3) 
\end{align}
we have
\begin{align} \label{the x=0 equation}
D_{t,0} f^M (s, X^s, \Theta^M_s )
   &= \wf^0(s,  D_{t,0} X^s)  \notag\\
  &\quad  +  \wf_{\ry}^0(s) \, D_{t,0} Y^M_s  +  \wf_{\rz}^0(s) D_{t,0} Z^M_s  +\wf_{\ru}^0(s)  \, D_{t,0} H^M_s
\end{align}
and furthermore, $D_{t,0} H^M_s =    \int_{\R_0}  \wh^0_{\ru}(s)  \,  D_{t,0}   U^M_s(v) \kappa (v)  \nu(dv). $ \medskip

To write the counterpart to    the above relation for a.a. $(t,x) \in [0,T]\times \R_0$  we use that there exist, by Corollary 
\ref{malliavin-derivative-explicite} and the mean value theorem, some measurable functions 
  $\eta :\Omega\times [0,T]\times \R\times [0,T] \to  [0,1]^{3}$, $\eta=\eta(\omega,t,x,s),$ 
  and $\vartheta:\Omega\times [0,T]\times \R_0\times [0,T]\times \R_0 \to [0,1]$ with $\vartheta=\vartheta(\omega,t,x,s,v)$ such that  
  (we use the abbreviation $ \wf^x_{ x_k}(s)$ etc to indicate that we do not have a partial derivative of $f$ but a partial derivative of  the truncated  $f^M$  which is taken for the jump case at an intermediate point.)

   \begin{align} \label{f-variables}
  \wf^x(s, D_{t,x}X^s) &:= f^M(s, X^s +  D_{t,x} X^s, \Theta^M_s + D_{t,x}\Theta^M_s)  \notag \\& \quad -   f^M(s, X^s,   \Theta^M_s+ D_{t,x}\Theta^M_s) \notag \\
  \wf_{\rm y}^x(s)&:= \partial_{y}  f^M(s, X^s, Y^M_s + \eta_1 D_{t,x} Y^M_s\!\!, (Z^M_s\!\!, H^M_s) + D_{t,x}  (Z^M_s, H^M_s)  ) \notag\\
   \wf_{\rm z}^x(s)&:= \partial_{z}  f^M(s, X^s, Y^M_s, Z^M_s +\eta_2 D_{t,x} Z^M_s, H^M_s + D_{t,x}  H^M_s  ) , \notag\\
    \wf_{\rm u}^x(s)&:= \partial_{u}  f^M(s, X^s, Y^M_s, Z^M_s , H^M_s +\eta_3 D_{t,x}  H^M_s  ) , \notag\\
    \wh_{\rm u}^x(s)&:= \partial_u h(s, U^M_s(v) + \vartheta D_{t,x}U^M_s(v))
    \end{align}
  
  By using a telescopic sum for $x  \in \R_0$  and \eqref{the x=0 equation} for $x=0$ we have
  \begin{align}  \label{mean-value-term-f}
     &D_{t,x} f^M(s, X^s,  \Theta^M_s) = \wf^x(s,  D_{t,x} X^s)+\wf^x_{\ry}(s)D_{t,x}Y^M_s \notag\\
     & \quad \quad \quad+\wf^x_{\rz}(s)D_{t,x}Z^M_s+\wf^x_{\ru}(s)      \int_{\R_0} \wh^x_{\ru}(s) \, D_{t,x}   U^M_s(v) \kappa (v)  \nu(dv).
  \end{align}

 We rewrite \eqref{Z-as-DY} as
\begin{align} \label{U-pi-M-ep-repr}  
 & Z^M_t  \one_{x = 0} +  U^M_t(x) \one_{x \neq 0}  \notag \\
    & =  D_{t,x} g^M(X^T) +  \int_t^T  \wf^x(s,  D_{t,x} X^s) \notag  +  \wf^x_{\ry}(s) D_{t,x}Y^M_s ds  \notag \\
  &-     \int_t^T D_{t,x}   Z^M_s   (dW_s-  \wf^x_{\rz} (s)ds)  \notag  \\
  & -\int_{{]t,T]}\times{\R_0}} D_{t,x}    U^M_s(v) \, (\tilde \nN(ds,dv)-
    \widehat f^x_{\ru}(s)   \widehat h^x_{\ru} (s) \,\, \kappa (v) \,\,ds \nu(dv)  ),
  \end{align}
and  consider the adjoint equation to \eqref{Z-as-DY}
\begin{align} \label{the-process-Gamma}
\Gamma^x_{t,u}  &= 1 + \int_t^u   \widehat f^x_{\ry}(s) \Gamma^x_{t,s} ds +  \int_t^u  \wf^x_{\rz}(s) \Gamma^x_{t,s} dW_s \notag\\ 
& \quad+  \int_t^u   \int_{\R_0} \wf^x_{\ru} (s) \wh^x_{\ru}(s)   \,\, \kappa (v) \,\, \Gamma^x_{t,s-} \tilde \nN(ds,dv),  \,\,u \in [t,T]. 
\end{align}
  
  \begin{lem}  \label{coefficients-path}    Suppose Assumptions  \ref{X-assumptions} and \ref{Ypath-assumptions} hold and that  there exist constants   $a, b \geq 1,$ (possibly depending on $M$)   such that, for a.a. $t \in [0,T]$, $x \in \R_0$,
    $$ |Z^M_t|  \le a +  b |X^t|_\infty^{\tr}  \quad \text{ and }  \quad |U^M_t(x)| \le \kappa_{\rho}(x)(a+ b|X^{t}|_\infty^{\tr} ).$$ 
      \begin{enumerate}[(1)]
     \item \label{derivative-estimate}   Then it holds for a.a. $ s \in [0, t]$  (recall \eqref{the-H})
    \begin{align*} 
      |D_{s,x}Z^M_t| &\le  C(1+ a +  b |X^t|_\infty^{\tr}) , \quad x\neq 0,   \\  
      |D_{s,x} U^M_t(v)  |& \le \kappa_{\rho}(v) C(1+ a +  b |X^t|_\infty^{\tr}), \quad x\neq 0,  \\
       |D_{s,x} H^M_t|  &\le   C(1 + a^{m_2+1}+ b^{m_2+1}   |X^t|_\infty^{\tr (m_2+1)}) , \quad x\neq 0,  
            \end{align*}
     where the constants $C$ depend on $\kappa_{\rho}$ through $\kappa_{\rho,2}$ and $\kappa_{\rho,\infty}$.
      \item  \label{f-estimate}  For the expressions defined in  \eqref{fhat-zero-terms} and \eqref{f-variables}  we have for a.a. $t \in [0,s]$ 
      \begin{align*}
    |\wf^x(s, D_{t,x}X^s)|   &\le C (1+ |X^s|_\infty^{\tt r} ),\;\; \forall k=1,\cdots,N\,,   \\  
    |\wf^x_{\ry}(s)| &\le L_{f,{\ry}},  \\
    |\wf^x_{\rz}(s)| &  \le  C (1+  a^{\ell} +  b^{\ell} |X^s|_\infty^{{\tt r} \ell}),      \\
    | \wf^x_{\ru}(s)  \wh^x_{\ru} (s)| &\le  C (1+ a^\ell +   b^\ell |X^s|_\infty^{{\tt r} \ell}).
     \end{align*}
     \end{enumerate}
      \end{lem}
  \begin{proof} 
    For $x \neq 0,$  since $ Z^M_t,$  $U^M_t$   and $H^M_t$ are $\ftn_t$ -measurable where  $(\ftn_t)_{t\in [0,T]}$ is the augmented  
    natural filtration generated by $(\mathcal{L}_s )_{s\in [0,T]}$ by \cite[Theorem 3.4]{Steinicke} we can  represent   them for fixed $t$  as  functions of the  
    L\'evy process $ \mathcal{L}$, denoting  $\mathcal{L}^t :=(\mathcal{L}_{s \wedge t})_{s\in [0,T]} $:    
    $$Z^M_t = F_Z(  \mathcal{L}^t),   \quad U^M_t(\cdot) = F_U(\cdot, \mathcal{L}^t),  \text{ and }   \quad  H^M_t = F_H(  \mathcal{L}^t) \quad   a.s. $$ 
    Similarly, $a +  b |X^t|_\infty^{\tr} = F_X(  \mathcal{L}^t).$   Here $F_Z,F_H,F_X$  are  $\mathcal{G}_t$-measurable and $F_U$ is $\mathcal{B}(\R_0) \otimes \mathcal{G}_t$-measurable.
   The lemma below uses the assumptions on $X$ and provides the bounds for item \eqref{derivative-estimate}:
    
     \begin{lem}  \label{derivative-bounds}
    Assume $F:D([0,T]) \to \R$ is $\mathcal{G}_t$-measurable.
    Then, for a.a.  $s \in [0, t]$  and $x\neq 0$  and any  constants $A,B \ge 0$ it holds that
    $$ |F(\mathcal{L}^t)| \le A + B |X^t|_\infty^{\tr}   \,\, \text{ implies} \,\, |D_{s,x}  F(\mathcal{L}^t)| \le 2A +2 B |X^t|_\infty^{\tr} + BT^{\tr} e^{L_b \tr T}\kappa_{\rho,\infty}^{\tr} $$
     for $ \PP \otimes  \m$-a.e.~ $(\omega, s,x).$ 
     \end{lem} 
    \begin{proof}
   Since we have an a.s.~representation   $F_X( \mathcal{L}^t) =  A + B |X^t|_\infty^{\tr}$,  we get by our assumption that setting
    $$\Lambda  := \{ \tx \in  D[0,T] :   F( \tx^t)   >  F_X( \tx^t) \}$$
     leads to $\PP( \mathcal{L}\in \Lambda ) = 0. $  Then   by  Lemma \ref{shiftthm},
     \begin{align} \label{shift-relation}
      \PP \otimes  \m \Big( (\omega, s,x): F(   \mathcal{L}^t +   x\one_{[s,t]})    > F_X(\mathcal{L}^t  +  x\one_{[s,t]} )  \Big ) = 0.
      \end{align} 
    By Lemma \ref{functionallem}, Corollary \ref{malliavin-derivative-explicite} and   Proposition \ref{bounds-for-DX},
      \begin{align*}
     F_X( \mathcal{L}^t  +  x\one_{[s,t]} ) &= D_{s,x} (A +  B |X^t|_\infty^{\tr})  +  A +  B |X^t|_\infty^{\tr}  \\
     &= B| D_{s,x}  X^t   +X^t   |_\infty^{\tr}  - B |X^t|_\infty^{\tr}   +  A +  B |X^t|_\infty^{\tr} \\
      &=  A + B| D_{s,x}  X^t   +X^t |_\infty^{\tr} \\
      &\le   A + B |X^t|_\infty^{\tr} + B T^{\tr} e^{L_b \tr T} \kappa_\rho(x)^{\tr}.
     \end{align*} 
     Since we have $D_{s,x}  F( \mathcal{L}^t)   +F(\mathcal{L}^t)  =  F(  \mathcal{L}^t +   x\one_{[s,t]}) $  it follows by \eqref{shift-relation} that
      \begin{align*}
D_{s,x}  F( \mathcal{L}^t)  +F( \mathcal{L}^t) 
   &  \le  F_X(\mathcal{L}^t  +  x\one_{[s,t]} )   \\
   &\le   A + B |X^t|_\infty^{\tr} + B T^{\tr} e^{L_b \tr T} \kappa_\rho(x)^{\tr}. 
        \end{align*} 

         Since we also have by assumption that
         $ - F( \mathcal{L}^t)  \le    A + B |X^t|_\infty^{\tr}= F_X(\mathcal{L}^t)  $
         we can repeat the above arguments to get
\begin{align*}      - D_{s,x}  F( \mathcal{L}^t) -F(\mathcal{L}^t) &= -F(\mathcal{L}^t  +  x\one_{[s,t]} ) \le F_X( \mathcal{L}^t  +  x\one_{[s,t]} )
         \end{align*}
       we arrive eventually at
   \begin{align*}
   |D_{s,x}  F(\mathcal{L}^t) | &\le |D_{s,x}  F( \mathcal{L}^t)   +F(\mathcal{L}^t)| + |F( \mathcal{L}^t)| \\
   &\le 2( A + B |X^t|_\infty^{\tr} )+ B T^{\tr} e^{L_b \tr T} \kappa_\rho(x)^{\tr}.
        \end{align*} 
  \end{proof}
    \bigskip 
    
    We continue with the proof of Lemma \ref{coefficients-path}.  The estimates in  item \eqref{derivative-estimate}  follow now immediately from the assumed bounds on $Z^M_t$ and   $U^M_t.$  
  By Assumption  \ref{Ypath-assumptions} \eqref{h-assumption}
 we have  $|h(s,u)| \le  \left (c+ \frac{\gamma}{2} |u|^{m_2} \right )|u| $ so that  the assumption  $|U^M_t(x)| \le \kappa_{\rho}(x)(a+ b|X^{t}|_\infty^{\tr} )$ implies
      \begin{align} \label{inner-product-path}
      &  |H^M_t |  = \left |   \int_{\R_0}  h(s, b_M ( U^M_s(v)))  \,  \kappa (v)  \nu(dv)\right | \notag \\
      &\le \int_{\R_0}  \left (c+ \frac{\gamma}{2} |U^M_s(v)|^{m_2} \right )|U^M_s(v)|  \,  \kappa (v)  \nu(dv)  \notag\\
      &\le   \int_{\R_0}  \kappa(v)  \kappa_{\rho}(v) \nu(dv)  \,\,  \left (c+ \frac{\gamma}{2} \left [\kappa_{\rho,\infty}(a+ b|X^{t}|_\infty^{\tr} )\right ]^{m_2}\right ) (a+ b|X^t|_\infty^{\tr} )   \notag \\
        & \le C(1 + a^{m_2+1}+ b^{m_2+1}   |X^t|_\infty^{\tr (m_2+1)} ),
        \end{align}  
     where  $C$ depends on $\gamma,\kappa_{\rho,2}$ and $\kappa_{\rho,\infty}$.

    \bigskip      
         
   We show  item \eqref{f-estimate}.    For $x \in \R_0,$ by  Assumption \ref{Ypath-assumptions} and Proposition \ref{bounds-for-DX}, 
   \eqref{f-variables} and item \eqref{derivative-estimate} we get  
    \begin{align*}
    |\wf^x(s, D_{t,x}X^s)| &
    = |f^M(s, X^s +  D_{t,x} X^s, \Theta^M_s + D_{t,x}\Theta^M_s) \\&-   f^M(s, X^s,   \Theta^M_s+ D_{t,x}\Theta^M_s)|\\
    &\le \left (c+ \frac{\beta}{2}( 2|X^s|_\infty^{\tt r} +  | D_{t,x}X^s|_\infty^{\tt r}) \right) | D_{t,x}X^s|_\infty^{\tt r} \\
    &\le  C (1+  |X^s|_\infty^{\tt r} ) \\
    & \text{ with } \,\, C= C(L_b, \tr, T, \kappa_{\rho,\infty}, \beta)\\
    |\wf^x_{\ry}(s)| &\le L_{f,{\ry}} \\
    |\wf^x_{\rz}(s)| &\le   \left (c+\gamma 2^{\ell} (|Z^M_s|^\ell + | D_{t,x}Z^M_s  |^\ell) \right )
    \le  C (1+  a^{\ell} +  b^{\ell} |X^s|_\infty^{{\tt r}\ell}), \\
    |\wf^x_{\ru}(s)| &\le \left (c+\gamma 2^{m_1} \Big (\big |H^M_s\big|^{m_1}  +  \big|D_{t,x} H^M_s\big|^{m_1} \Big)\right ) \\
    |\wf^x_{\ru}(s) \wh^x_{\ru}(s) | &\le \left ( c +\gamma  2^{m_1} \Big( \big|H^M_s\big|^{m_1} + \big| D_{t,x}H^M_s\big|^{m_1}\Big ) \right ) \;\; \\
    &  \quad\times  ( c+ \gamma 2^{m_2} (|U^M_s(v)|^{m_2}  + |D_{t,x} U^M_s(v)|^{m_2}) ).
     \end{align*}
    
    Then  by \eqref{inner-product-path} and  item \eqref{derivative-estimate}
    \begin{align*}
     | \wf^x_u (s) \wh^x_u(s) | &\le C( c  + a^{(m_2+1)m_1}   + b^{(m_2+1)m_1}   |X^s|_\infty^{\tr(m_2+1)m_1} )\\
     & \quad \times (c+ (\kappa_{\rho}(v) C(1+ a +  b |X^s|_\infty^{\tr}) )^{ m_2})\\
     &\le  C\left (\!1+ a^{(m_1 + m_1m_2 +m_2)}\!+    b^{(m_1 + m_1m_2 +m_2)} |X^s|_\infty^{\tr (m_1 + m_1m_2 +m_2)} \!\right ) \\
      &\le  C\left (1+ a^{\ell}+    b^{\ell} |X^s|_\infty^{\tr \ell} \right ) 
    \end{align*}
    where $C$ depends on $\gamma,\kappa_{\rho,2}$ and $\kappa_{\rho,\infty}$.

For $x=0$ we get  the same bounds by similar computations.
    \end{proof}
    
For a.a.~$(s,x)  \in [0,T]\times\R$,  the process  $(G^x_{s,t})_{t\in[s,T]}$  defined by
   \begin{align}  \label{the-process-G}
    G^x_{s,t}  &= 1  +  \int_s^t \wf^x_{\rz}(r) \, G^x_{s,r} dW_r +  \int_{]s,t]\times \R_0} \wf^x_{\ru}(r)  \wh^x_{\ru}(r)   \,\, \kappa (v) \,\, G^x_{s,r} \tilde \nN(dr,dv),  
   \end{align} 
  is  a local martingale, but the next lemma shows that under our conditions  it is a true martingale.

\begin{lemma}  \label{unif-martingale}  Suppose Assumptions  \ref{X-assumptions} and \ref{Ypath-assumptions} hold, and that there  exist constants   $a, b \geq 1,$   (possibly depending on $M$)  such that for a.a.~$(t, x)  \in [0,T]\times \R_0$,
    $$ |Z^M_t|  \le a +  b  |X^t|^{\tr}_\infty  \quad \text{ and }  \quad |U^M_t(x)| \le \kappa_{\rho}(x)(a+ b |X^t|^{\tr}_\infty ). $$ 
   Then, the process  $(G^x_{s,t}) _{t\in [s,T]}$  given in \eqref{the-process-G} is a  uniformly integrable martingale and we have for the solution to \eqref{U-pi-M-ep-repr} the representation
  \begin{align*} & Z^M_t  \one_{x = 0} +  U^M_t(x) \one_{x \neq 0}  \notag \\
  &=\E^{x\,'}_t  \Big [ e^{\int_t^T \wf^x_{\ry} (r) dr}D_{t,x} g^M(X) + \int_t^T e^{\int_t^v \wf^x_{\ry} (r) dr}  \wf^x(v,  D_{t,x} X^v) dv \Big ],
  \end{align*}
  where $d {\PP}^{x \, '} =  G^x_{t,T}  d\PP.$ 
  \end{lemma}
  
  \begin{proof}   Since $(G^x_{s,t} )_{t\in [s,T]}$ is a locally square integrable martingale, by Novikov's condition given by \cite[Theorem 2.4]{Sokol} it suffices to check if
  $$\E \exp \left (\frac{1}{2} \int_s^T |\wf_{\rz}^x(u) |^2du+  \int_s^T   |\wf_{\ru}^x(u)|^2   \,\,\int_{\R_0} | \wh^x_{\ru}|^2 \kappa^2 (v) \nu(dv) du  \right )<\infty.$$
  This follows from Lemma \ref{exponential-bound} since by Lemma \ref{coefficients-path} we have
  \begin{align*}
   &\int_s^T \left ( \frac{1}{2}|\wf_{\rz}^x(u) |^2+  |\wf_{\ru}^x(u)|^2  \,\,\int_{\R_0} | \wh^x_{\ru}|^2 \kappa^2 (v) \nu(dv)\right ) du \\
   & \le   \int_s^T C^2 (1+a^\ell +  b^\ell |X^u|_\infty^{{\tt r}\ell} )^2 + C^2  (1+ a^\ell +   b^\ell|X^u|_\infty^{{\tt r} \ell})^2du \\
  & \le \int_s^T C'(1+ |X^u|_\infty) du
  \end{align*} 
  (note that by assumption we have $2 \tr \ell \le 1$   and  $ \kappa \in L^2(\nu)$). The second part of the Lemma follows from the proof of \cite[Theorem 3.4]{QuenezSulem}.
   It holds  $\Gamma^x_{s,t}  =  e^{\int_s^t  \wf^x_{\ry}(r)  dr }  G^x_{s,t}$ (see \eqref{the-process-G} and \eqref{the-process-Gamma}).
  \end{proof}

   Before stating the next result let us prove an auxiliary Lemma.
   \begin{lemma}\label{lem_cond_est_X_pi}
    Suppose Assumptions  \ref{X-assumptions} and \ref{Ypath-assumptions} hold, and that there exists constants  $a, b \geq 1,$  (possibly depending on $M$)  such that for a.a.~$(t, x)  \in [0,T]\times \R_0$, 
    $$ |Z^M_t|  \le a +  b |X^t|^{\tr}_\infty  \quad \text{ and }  \quad |U^M_t(x)| \le \kappa_{\rho}(x)(a+ b |X^t|^{\tr}_\infty ). $$
    For all $0\le s \le t \le T$ we have
    \begin{align*}
      \E^{x\, '}_s |X^t|^{\tt r}_\infty&\le C(1+a^{{\tt r }\ell}+b+|X^s|^{\tt r}_\infty).
    \end{align*}
  \end{lemma}
  
  \begin{proof}Since
    \begin{align*} 
    \E^{x\, '}_s|X^t|_\infty^{\tt r}\le |X^s |_\infty^{\tt r} +   \E^{x\, '}_s|X^t -X^s|_\infty^{\tt r},
    \end{align*}
  recalling $\tr <1$, it suffices to  estimate the term 
   $$   \E^{x\, '}_s|X^t -X^s|_\infty^{\tt r} =   \E^{x\, '}_s   \sup_{s\le u\le t} |X_u-X_s|^{\tt r}.$$
     The process  $X$   for $t \in [s,T]$ is given by
       \begin{align*}
       X_t  &= X_s + \int_s^t b(u,X_u) du +  \int_s^t \sigma(u) dW'_u  + \int_s^t  \sigma(u)  \wf^x_\rz (u) du \\
       & \quad + \int_{{]s,t]}\times{\R_0}}  \rho(s,v)   \tilde \nN'(du,dv)  +    \int_{{]s,t]}\times{\R_0}}   \rho(s,v)    \wf^x_\ru(u) \wh^x_\ru  (u) \,\, \kappa (v) \,\,du \nu(dv).  
       \end{align*}

     By It\^o's formula it holds for $0 \le s < t \le T$ that
       \begin{align*} 
     & (X_t-X_s)^2  \\ &= 2 \int_s^t  (X_u-X_s)( b(u,X_u) +  \sigma(u)  \wf^x_\rz(u) ) du    +  2 \int_s^t  (X_u-X_s)\sigma(u) dW'_u   \\ 
     & \quad +  \int_s^t  \sigma^2(u) du \\
       & \quad +  \int_{{]s,t]}\times{\R_0}} 2(X_u -X_s) \rho(s,v) +  \rho(s,v)^2  \tilde \nN'(du,dv)  \\ 
       & \quad +  2   \int_{{]s,t]}\times{\R_0}}  (X_u -X_s) \rho(s,v)    \wf^x_u \wh^x_u  \,\, \kappa (v) \,\,du \nu(dv)  \\ 
       &\quad+  \int_{{]s,t]}\times{\R_0}}  \rho(s,v) ^2\,\,du \nu(dv).
        \end{align*}
  To shorten the notation we will address the sum of the stochastic integrals on the r.h.s.~as  'local martingale'.  The two non-random terms are bounded by  $( K^2_\sigma+\kappa_{\rho, 2})(T-s).$
Using  the bounds for $|\wf^x_\rz (u) |$   and $ | \wf^x_\ru(u) \wh^x_\ru(u)|$ given in    Lemma \ref{coefficients-path}  we continue to estimate the remaining integrands as follows:
  \begin{align*}
&   | (X_u-X_s)( b(u,X_u) +  \sigma(u)  \wf^x_\rz (u))  |   \\& \quad + \left | \int_{{\R_0}}  (X_u -X_s) \rho(s,v)    \wf^x_\ru(u) \wh^x_\ru(u)  \, \kappa (v)  \nu(dv)  \right | \\
& \le C |X_u-X_s|  (|X_u|+1+  a^{\ell} +  b^{\ell} |X^u|_\infty^{{\tt r}\ell}),
  \end{align*}
where $C$ depends on $L_b, K_\sigma, K_b,  \kappa_{\rho,2} $  and $\kappa_{\rho,\infty}.$ Since ${\tt r}\ell\le \frac{1}{2}$ we may  increase the exponent of $|X^u|_\infty^{{\tt r}\ell}$  to $\frac{1}{2}$ if we add $1$ and then  we apply the Cauchy-Schwarz inequality to get
$$  b^{\ell} |X^u|_\infty^{{\tt r}\ell} \le  b^{\ell} (1+  |X^u|_\infty^ \frac{1}{2}) \le b^{\ell} + \frac{b^{2\ell}}{2} + \frac{ |X^u|_\infty}{2} \le  2b^{2\ell}  +  |X^u|_\infty. $$
This gives
 \begin{align*} 
&|X_u-X_s|  (|X_u|+1+  a^{\ell} +  b^{\ell} |X^u|_\infty^{{\tt r}\ell}) \\
& \le  |X^u-X^s|_\infty (2 |X^u|_\infty +1+  a^{\ell} +  2b^{2\ell} )  \\
& \le  |X^u-X^s|_\infty (2 |X^u-X^s|_\infty  +  2 |X^s|_\infty+1+  a^{\ell} +  2b^{2\ell})  \\
&\le 3 |X^u-X^s|_\infty^2 +  (|X^s|_\infty+1+  a^{\ell} +  b^{2\ell})^2.
 \end{align*}
 Summarising, we have
  \begin{align*}
    (X_t-X_s)^2  \le    6 C \int_s^t  \sup_{s\le v\le u} |X_v-X_s|^2  du   + H  + \,\, \text{ local martingale}
     \end{align*} 
 where 
 \begin{align*} 
 H &= \left (K^2_\sigma  +\kappa_{\rho, 2}   + C (|X^s|_\infty+1+  a^{\ell} +  b^{2\ell})^2 \right )T.
   \end{align*}
 Then, by   relation \eqref{stochGronwall}  of Theorem \ref{Sarah2}  (for $0< p= \frac{\tt r}{2}<1$ )
 \begin{align*}
  \E^{x\, '}_s\sup_{s \le u\le t}|X_u-X_s |^{\tt r} &\le   \frac{1}{(1-p) }  H^{{\tt r}/2} \,\,   e^{\frac{p}{1-p}  6C (T- s)}\\
  &\le C'(1+a^{{\tt r }\ell}+b+|X^s|^{\tt r}_\infty),
  \end{align*}      
  where  we used that $b^{2{\tt r} \ell} \le b.$ The constant $C'$ depends on $L_b, K_\sigma, K_b,  \kappa_{\rho,2}, \kappa_{\rho,\infty}.$ and  ${\tt r}$.
\end{proof}

Lemma \ref{estimates-depending-on-M} presented below gives us some bounds on $Z^M$ and $U^M$ that cannot be used directly to prove the existence of a solution for \eqref{fbsde} since they strongly depend on $M$ through the constants $a=a_M$ and $b=b_M$. It is however  an important  step in the proof of Proposition \ref{bounds-independend-from-M}.
   
   \begin{lem} \label{iterated-estimate}\label{estimates-depending-on-M}  Suppose Assumptions  \ref{X-assumptions} and \ref{Ypath-assumptions} hold.  Then the following holds for a solution $(Y^M, Z^M, U^M)$ to \eqref{BSDE-truncated}: Assume that there  exist constants ${a, b} \geq 1$ which may depend on $M $  such that for a.a.  
     $(t, x)  \in [0,T]\times \R_0$, 
    $$ |Z^M_t|  \le a +  b  |X^t|^{\tr}_\infty  \quad \text{and} \quad  |U^M_t(x)| \le \kappa_{\rho}(x)(a+ b  |X^t|^{\tr}_\infty ).$$
    Then  there is a $C>0$ which does not depend on $a$, $b$, $M$ such that
     $$ |Z^M_t | \le  C(( 1+ a^{r \ell}+b) +  | X^t |^{\tt r}_\infty  )  $$
    and  
      $$ |U^M_t(x)| \le  C \kappa_\rho(x)   (( 1+ a^{r \ell}+b) + | X^t |^{\tt r}_\infty  ). $$
    
     \end{lem}
    
    \begin{proof} 
    We start by proving the upper bound for $|U^M_t(x)|$ by using Lemma \ref{unif-martingale}.
    
    From  Assumption \ref{Ypath-assumptions} and  Proposition \ref{bounds-for-DX}
    we derive that  for $x \neq 0$
    \begin{align} \label{U-estimate-termA}
    & \Big |  \E^{x\, '}_t \Big [ e^{\int_t^T \wf^x_y (r) dr}D_{t,x} g^M(X)  \Big ] \Big |  \notag \\
    & \le  e^{TL_{f,\ry}}   \E^{x\, '}_t \Big [ \left (c+  \alpha (|X|_\infty^{\tt r} + |D_{t,x} X|_\infty^{\tt r}) \right )|D_{t,x} X|_\infty \Big ]  \notag \\
     & \le  e^{TL_{f,\ry}}  T e^{L_b T} \kappa_\rho(x)   \E^{x\, '}_t\left (c+  \alpha (|X|_\infty^{\tt r} + | e^{L_b T} \kappa_\rho(x) |^{\tt r}) \right ) \notag\\
    &\le \kappa_\rho(x)C ( 1 +  \E^{x\, '}_t|X|_\infty^{\tt r} ).
     \end{align}

     Hence from \eqref{U-estimate-termA} and Lemma \ref{lem_cond_est_X_pi} we get that there is a $C>0$  such that we have
    \equa
     && \Big |  \E^{x\, '}_t \Big [ e^{\int_t^T \wf^x_y(r)  dr}D_{t,x} g^M(X) \Big ] \Big |  \le C \kappa_\rho(x)  ( 1+ a^{r \ell} +b + | X^t|_\infty^{\tt r}),
     \tion
     
    Applying  Lemma \ref{coefficients-path}   we get as in \eqref{U-estimate-termA} that
     \begin{align*}  &\Big |  \E^{x\, '}_t \Big [ \int_t^T e^{\int_t^v \wf^x_{\ry}(r)  dr} \wf^x(v, D_{t,x}X^v) dv  \Big ]  \Big |\\
     & \le C \int_t^T  \E^{x\, '}_t \Big [ (1 +e^{L_b \tt r T} \kappa_{\rho}(x)^{\tt r}  +| X^v  |_\infty^{\tt r}) \, e^{L_bT}\kappa_{\rho}(x) \,\,  \Big ]  dv .
     \end{align*}
     
   By  Lemma \ref{lem_cond_est_X_pi} we have
     \begin{align*}
      &\Big |  \E^{x\, '}_t \Big [ \int_t^T e^{\int_t^v \wf^x_{\ry}(r)  dr} \wf^x(v, D_{t,x}X^v) dv  \Big ]  \Big |\le C \kappa_\rho(x)  ( 1+ a^{r\ell} + b + | X^t |_\infty^{\tt r}).
    \end{align*}
     
    \end{proof}
    
    We have now all the tools needed in order to prove Proposition \ref{bounds-independend-from-M}.
    
    \begin{proof}[Proof of Proposition \ref{bounds-independend-from-M}]   \label{proof-of-first-main} Let us start by proving \eqref{Z-and-U-pi-bound}. By Lemma  \ref{first-estimates-depending-on-M-pi-eps}   
    there exists an $a_0>0$ potentially depending on $M$ such that, for a.a.  $(t,x) \in [0,T] \times \R$,
      \begin{align}\label{step0}
    &|Z^M_t|  \le a_0,\notag\\
    &|U^M_t(x)| \le a_0 \kappa_{\rho}(x) .
      \end{align}
     By Lemma \ref{iterated-estimate} we have
      $$ |Z^M_t | \le  C( 1+ a_0^{r \ell}) + C| X^t |_\infty^{\tt r}  $$
      $$ |U^M_t(x)| \le   \kappa_\rho(x)   (C( 1+ a_0^{r \ell}) + C| X^t |_\infty^{\tt r}  ), $$
      where $C$ does not depend on $a_0$ and $M$.
      Let $a_1:=C(1+ a_0^{r \ell}).$ Applying again Lemma \ref{iterated-estimate} gives
      $$ |Z^M_t | \le  C( 1+ a_1^{{r \ell}} +C)) + C| X^{\pi,t} |_\infty^{\tt r} \le (C+1)^2(1+a_1^{r \ell})+ C| X^{\pi,t} |_\infty^{\tt r} $$
      $$ |U^M_t(x)| \le   C\kappa_\rho(x)   (( 1+ a_1^{r \ell}+C) + | X^{\pi,t} |_\infty^{\tt r}  )\le \kappa_\rho(x) ((C+1)^2(1+a_1^{r \ell})+ C| X^{\pi,t} |_\infty^{\tt r}). $$
    Let $a_{n+1}:= (C+1)^2(1+a_n^{rl})$. Lemma \ref{iterated-estimate} gives for all $n$
    $$ |Z^M_t | \le (C+1)^2(1+a_{n+1}^{r\ell})+ C| X^{\pi,t} |_\infty^{\tt r} $$
      $$ |U^M_t(x)|\le \kappa_\rho(x) ((C+1)^2(1+a_{n+1}^{r \ell})+ C| X^{\pi,t} |_\infty^{\tt r}). $$
      We can remark that $a_{\infty}:=\lim_{n \rightarrow +\infty}a_n$ exists since $r\ell<1$, and it depends only on $C$ and $r\ell$. Then we just have to set  $b_{\infty}:=C$ in order to get the announced result.\\
 We now show \eqref{Y-pi-bound}. Since 
  \begin{align*}
  Y^M_t &= \E_t g^M(X) + \E_t   \int_t^T  f^M\left( s, X^s, \Theta^M_s \right)ds
\end{align*}
where  $\Theta^M_s:= (Y^M_s, Z^M_s,  H^M_s )$, we get  from \eqref{Z-and-U-pi-bound} and Assumption \ref{Ypath-assumptions}
 \begin{align*}
  |Y^M_t |&\le | g^M(0)| +   \E_t |g^M(X) -g^M(0)|  \\
  & \quad +   \int_t^T | f^M\left( s, 0,0,0,0) \right)|ds  \\
  &\quad+\E_t   \int_t^T | f^M\left( s, X^s, \Theta^M_s \right)-f^M \left( s, 0,0,0,0) \right)| ds  \\
  &\le C( g(0))  +  C \E_t (1+  |X|_\infty^{{\tt r} +1} ) \\
  & \quad + C \int_t^T |f(s,0,0,0,0)|ds +\E_t   \int_t^T  \big [C(1+  |X^t|_\infty^{{\tt r} +1} )  + L_{f, \ry} |Y^M_s | +  \\
   & \quad + C(1+  |Z^M_s |^{l+1} )  + C(1+    | H^M_s |^{m_1+1} ) ]ds \\
   &\le C + C \E_t  |X|_\infty^{{\tt r} +1} +  C  \int_t^T  \E_t   |X^s|_\infty^{{\tt r} +1}  ds + L_{f, \ry} \E_t   \int_t^T   |Y^M_s | ds \\
   &\le C + C  |X^t|_\infty^{{\tt r} +1}  + C\E_t  |X  - X^t  |_\infty^{{\tt r} +1}      \\
   & \quad + C  \int_t^T  \E_t   |X^s  - X^t   |_\infty^{{\tt r} +1}  ds + L_{f, \ry} \E_t   \int_t^T   |Y^M_s | ds.
  \end{align*}
 
   Then, by similar computations as in the proof of Lemma \ref{lem_cond_est_X_pi} (here we use  $\E$ instead of $\E'$) one gets  that 
  $$|Y^M_t | \le C + C  |X^t|_\infty^{{\tt r} +1} +  L_{f, \ry} \E_t   \int_t^T   |Y^M_s | ds $$
  and the assertion follows by Gronwall's lemma. 
\end{proof}
       
\section{A priori estimate and proof of the main result}\label{sec: finish}

Before we show the main theorem, we prove a stability result well shaped for our setting.  We introduce the notation 
$$ f_h(s,\bx^s, y,z, {\bf u}) := f \left (s,\bx^s, y,z,  \int_{\R_0}  h(s,  {\bf u}(v)) \kappa (v)  \nu(dv) \right).  $$

  \begin{lem}\label{lem:apriori}
   Let $(\xi,f,h), (\xi',f',h')$ be data for BSDEs with $\xi,\xi'\in L^2$ and $f,f',h,h'$ satisfy Assumption \ref{Ypath-assumptions} ({\tt iii}-{\tt vi}). 
  Assume two respective solutions $(Y,Z,U), (Y',Z',U')$ exist 
    and are such that there are constants $a,b>0$
  $$  |Z_t |\vee |Z'_t| \le a + b |X^t|^{\tt r}_\infty  \quad\text{and}\quad  |U_t(x) | \vee|U'_t(x)| \le \kappa_{\rho}(x)(a + b |X^t|^{\tt r}_\infty ).$$
  Then there is a constant $C=C(a,b,c,T,L_{f,\ry},\gamma,{\tt r},l, \kappa_{\rho,\infty}, \|\kappa\|_{L^2(\nu)})>0$   such that for all $n\in \NN$,
  \begin{align}\label{apriori_est}
  &\EE|Y-Y'|_\infty^2 + \EE\int_0^T|Z_t- Z_t'|^2dt +\EE\int_0^T \|U_t-U_t'\|_{L^2(\nu)}^2dt\notag \\
  &\leq Ce^{Cn}\bigg(\EE|\xi-\xi'|^2+\EE\bigg(\int_0^T|f_h(t,Y_t,Z_t,U_t)-f'_{h'}(t,Y_t,Z_t,U_t)|dt\bigg)^2\bigg) \notag\\
  &\quad+e^{-n}C\EE \left[e^{(T+1)^2|X|_\infty)}\right].
   \end{align}
    Moreover, the difference of the generators may also be taken using the solution $(Y',Z',U')$.
  \end{lem} 
   
  \begin{proof}
   We denote differences of the solutions by $\Delta Y_t:=Y_t-Y'_t$ etc.~and the one of the terminal conditions by $\Delta\xi:=\xi-\xi'$.
  We will here use the notation $\Theta_t: =( Y_t,Z_t,U_t)$ and $\Theta'_t :=( Y'_t,Z'_t,U'_t).$  
  
  By standard arguments using It\^o's formula for $|\Delta Y_t|^2$, Doob's $L^p$- and Young's inequalities, we find that one obtains $C>0$ (we will keep the constant's name $C$ in all subsequent estimates, but its value may change as we proceed) such that for all $t\in [0,T]$,
  \begin{align}\label{eq:apriori-1}
  &\EE\sup_{s\in[t,T]}|\Delta Y_s|^2+\EE\int_t^T|\Delta Z_s|^2ds +\EE\int_t^T \|\Delta U_s\|_{L^2(\nu)}^2ds\\
  &\leq C\EE\bigg[|\Delta \xi|^2+\bigg(\int_t^T |\Delta Y_s||f_h(s,\Theta_s)-f'_{h'}(s,\Theta'_s)|ds\bigg)\bigg].\nonumber
  \end{align}
  
  It holds using Assumption \ref{Ypath-assumptions} on $f'$ and $h'$ that
  $$|f'_{h'}(s,\Theta_s)-f'_{h'}(s,\Theta'_s)|\le C\Big (|\Delta Y_s|+(1+|X^s|_\infty^{\tr l} )\Big (|\Delta Z_s|+\|\Delta U_s\|_{L^2(\nu)} \Big).$$
  Here the bound $C (1+|X^s|_\infty^{\tr l}) \|\Delta U_s\|_{L^2(\nu)}^2$ is coming from Assumption \ref{Ypath-assumptions} (v), the bound on $U$ and the fact that 
  \begin{align*}
   &\left| \int_{\R_0}  (h(s, U_s(v)) -  h(s, U'_s(v)) )  \kappa (v)  \nu(dv)\right | \\
   &\le C(1+ |X^s|_\infty^{\tr m_2} )     
   \int_{\R_0} |U_s(v) - U'_s(v)|  \kappa (v)  \nu(dv) \\
    &\le C(1+ |X^s|_\infty^{\tr m_2}
     ) \, \|\kappa\|_{L^2(\nu)}\,\|\Delta U_s\|_{L^2(\nu)} .
      \end{align*}
   Now let us introduce $Q_{n,t}=\{(\omega,t): T\sup_{0\le s \le t}|X_s|>n\}$. By splitting up the domain and using Assumption \ref{Ypath-assumptions} and the bounds in $X$ for $Z$ and $U$, we obtain 
  \begin{align} \label{Y-timesDiff-f}
  &|\Delta Y_s||f_h(s,\Theta_s)-f'_{h'}(s,\Theta'_s)|  \notag\\
  &\leq C|\Delta Y_s|\Big (|\Delta Y_s|+(1+ |X^s|_\infty^{\tr  l})\Big (|\Delta Z_s|+\|\Delta U_s\|_{L^2(\nu)} \Big) +|(f_h -f'_{h'})(s,\Theta_s)|\Big )\nonumber\\
  &\leq C(|\Delta Y_s|^2 +|\Delta Y_s||(f_h -f'_{h'})(s,\Theta_s)|)\nonumber\\
  &\quad+C|\Delta Y_s|(1+   |X^s|_\infty^{\tr  l})(|\Delta Z_s|+\|\Delta U_s\|_{L^2(\nu)})(\one_{Q^c_{n,s}}+\one_{Q_{n,s}})\nonumber\\
  &\leq C(|\Delta Y_s|^2 +|\Delta Y_s||(f_h -f'_{h'})(s,\Theta_s)|)\nonumber\\
  &\quad+C|\Delta Y_s|(1+ |X^s|_\infty^{ \tr l} )  ( |\Delta Z_s|+\|\Delta U_s\|_{L^2(\nu)})\one_{Q^c_{n,s}} \nonumber\\
  &\quad + C|\Delta Y_s|(1+|X^s|_\infty^{{1+ \tt r}})\one_{Q_{n,s}}.
  \end{align}
  
  With the help of Young's inequality $yz\leq \frac{cy^2}{2}+\frac{z^2}{2c}$ for arbitrary $c>0$, bounding $|\Delta Y|$ by its supremum, we may estimate the term
  \begin{align}\label{eq:apriori-two}
  &|\Delta Y_s||f_h(s,\Theta_s)-f'_{h'}(s,\Theta'_s)|\\
  &\leq C|\Delta Y_s|^2+\frac{3Cc}{2}|\Delta Y_s|^2 (1+|X^s|_\infty^{2rl})\one_{Q_{n,s}^c}+\frac{C}{2c}(|\Delta Z_s|^2+\|\Delta U_s\|_{L^2(\nu)}^2)\nonumber\\
  &\quad+|X^s|_\infty^{2{\tt r}+2}\one_{Q_{n,s}}+\sup_{r\in [t,T]}|\Delta Y_r||(f_h -f'_{h'})(s,\Theta_s)|.\nonumber
  \end{align} 
  Note that on $Q_{n,s}^c$, we can estimate $1+|X^s|_\infty^{2lr}\leq 1+n.$
  With the above estimates, using again Young's inequality and choosing $c$ in an adequate way, we infer from \eqref{eq:apriori-1} that
  \begin{align*}
  &\EE\sup_{s\in[t,T]}|\Delta Y_s|^2+\EE\int_t^T|\Delta Z_s|^2ds +\EE\int_t^T \|\Delta U_s\|_{L^2(\nu)}^2ds\\
  &\leq C\EE\bigg[|\Delta \xi|^2+(n+2)\int_t^T|\Delta Y_s|^2ds\nonumber \\
  &\quad\quad+\bigg(\int_t^T|(f_h -f'_{h'})(s,\Theta_s)|ds\bigg)^2+\int_t^T|X^s|_\infty^{2 \tr +2}\one_{Q_{n,s}}ds\bigg].\nonumber
  \end{align*} 
  So, replacing $|\Delta Y_s|^2$ on the right hand side by $\sup_{r\in [s,T]}|\Delta Y_r|^2+\int_s^T(|Z_r|^2+\|U_r\|^2)dr$, we can use Gronwall's inequality to arrive at
  
  \begin{align*}
  &\EE\sup_{s\in[t,T]}|\Delta Y_s|^2+\EE\int_t^T|\Delta Z_s|^2ds +\EE\int_t^T \|\Delta U_s\|_{L^2(\nu)}^2ds\\
  &\leq Ce^{(n+2)(T-t)}\EE\bigg[|\Delta \xi|^2+\bigg(\int_t^T |(f_h -f'_{h'})(s,\Theta_s)|ds\bigg)^2\\
  &\quad \quad\quad+\int_t^T|X^s|_\infty^{2\tr+2}\one_{Q_{n,s}}ds\bigg].
  \end{align*}
  Since
  \begin{align*}
    e^{nT} |X^s|_\infty^{2\tr+2}\one_{Q_{n,s}}&\le e^{T^2 \sup_{0\le u \le s}|X_u| } (1+ |X^s|_\infty^{3})\one_{Q_{n,s}}\\
    &\le e^{T^2 \sup_{0\le u \le s}|X_u| } 2e^{|X^s|_\infty}e^{-n}e^{T \sup_{0\le u \le s}|X_u|}\\
    &\le e^{-n}2e^{(T+1)^2 |X|_\infty },
  \end{align*}
 we conclude that there is a $C>0$ such that
  \begin{align*}
  &\EE\sup_{s\in[t,T]}|\Delta Y_s|^2+\EE\int_t^T|\Delta Z_s|^2ds +\EE\int_t^T \|\Delta U_s\|_{L^2(\nu)}^2ds \notag\\
  &\leq Ce^{Cn}\EE\bigg[|\Delta \xi|^2+\bigg(\int_t^T |(f_h -f'_{h'})(s,\Theta_s)|ds\bigg)^2\bigg] \notag\\
  &\quad +e^{-n}C \EE \left[ e^{(T+1)^2) |X|_\infty }\right]
  \end{align*}
  and we just have to take $t=0$ in the previous inequality.
  \end{proof}

\subsection{Proof of Theorem \ref{ex-and-unique-sol} }
  The uniqueness follows from Lemma \ref{lem:apriori}. It remains to show existence. To that end, consider the terminal condition $\xi^M :=g^M(X)$ and the generator $f^M$  as given in  \eqref{data-cut-off}. By Proposition \ref{bounds-independend-from-M},  a solution $(Y^M,Z^M,U^M)$ exists which obeys the condition \eqref{bd-cond-x}  with the same $a,b,c$ for any $M$. Let us show by  Lemma \ref{lem:apriori}   that 
  $(Y^M,Z^M,U^M)_{M=1}^\infty$ is a Cauchy sequence in $\mathcal{S}^2\times L_2(W)\times L_2(\tilde{\nN})$. 
   We note that  by choosing  $n$ large  the third term of the r.h.s.~of \eqref{apriori_est} can be made arbitrarily  small.  
  So it suffices to show that  the other two expressions on the r.h.s.~of \eqref{apriori_est} are small for  large $M, \ov{M}$.    We show that  $\xi^M$ converges to $\xi^{\ov{M}}$ in $L^2.$
    By  Hölder's and Markov's inequality and Lemma \ref{exponential-bound}, if $M\le \ \ov{M},$
   \begin{align*} 
  \EE|\xi^M- \xi^{\ov{M}}|^2&=\E|g^M(X)- g^{\ov{M}}(X)|^2\\
    &\le \E\Big (( c+\frac{\alpha}{2}(|X^M|_\infty^{\tt r}+|X^{\ov{M}}|_\infty^{\tt r}))  |X^M- X^{\ov{M}}|_{\infty}  \Big)^2\\
    & \le \E\Big (( c+\alpha|X|_\infty^{\tt r}) |X|_\infty\one_{|X|_\infty\ge M}  \Big)^2\\
    &\le C( 1+ ( \E |X|_\infty^{4({\tt r}+1)})^{\frac{1}{2}} )  \big (\E \one_{|X|_\infty\ge M}  \Big)^{\frac{1}{2}}\le C' \frac{1}{M}.
    \end{align*}

 We continue with   
    
    \begin{align} \label{the-int}  \EE\bigg(&\int_0^T|f^M_h( s, X^s, \Theta^M_s) -f^{\ov{M}}_{h'}( s,X^s,\Theta^M_s) |ds\bigg)^2,
    \end{align}
  where 
 $$ f_h^M(s,\bx, y,z, {\bf u}) := f\!\! \left (\!s,b_M(\bx), y, b_M(z),  b_M \Big( \!\int_{\R_0}  h(s, b_M ( {\bf u}(v)) ) \kappa (v)  \nu(dv) \Big )\!\right).  $$
 and   
    $$ f_{h'}^{\ov{M}}(s,\bx^s\!, y,z, {\bf u}) := f \!\!\left (\!s,b_{\ov{M}}(\bx), y,b_{\ov{M}}(z), b_{\ov{M}} \Big( \int_{\R_0}  h(s, b_{\ov{M}} ( {\bf u}(v))) \kappa (v)  \nu(dv) \! \Big ) \! \right).  $$
    
Since Assumptions \ref{Ypath-assumptions}-(i) and (ii) are of the same type, the $x$ difference  of $f$ behaves like $\xi^M-\xi^{\ov{M}}$ so we get  for $M\le \ \ov{M}$ the upper bound 
    \begin{align*}
   & |f_h( s, b_M( X^s), \Theta^M_s) -f_h( s,b_{\ov{M}}(X^s),\Theta^M_s) | \\ 
   & \le     ( c+\frac{\beta}{2}(|X^M|_\infty^{\tt r}+|X^{\ov{M}}|_\infty^{\tt r})) \,  |X^M- X^{\ov{M}}|_{\infty}  \\
   & \le  ( c+\beta|X|_\infty^{\tt r}) |X|_\infty\one_{|X|_\infty\ge M}.
     \end{align*}
 Similarly,  we can treat $ f_h^M - f^{\ov{M}}_{h'}$ concerning $ Z_s^M$    and  $U_s^M.$  
For example,  by   Assumption \ref{Ypath-assumptions}(iv) and  Proposition \ref{bounds-independend-from-M},
 \begin{align*}
&   \left (c+ \frac{\gamma}{2} (| b_M( Z_s^M) |^\ell + |b_{\ov{M}}(  Z_s^M)   |^\ell) \right )  | b_M( Z_s^M)  - b_{\ov{M}}(  Z_s^M)| \\
 & \le   \left (c+ \gamma ( a_\infty +  b_\infty |X^t|_\infty^{\tr} )^\ell \right )  (a_\infty +  b_\infty |X^t|_\infty^{\tr}   ) \, \one_{(a_\infty +  b_\infty |X^t|_\infty^{\tr}) \ge M}.
    \end{align*}

  Continuing like for $\xi^M$ we see that the integral term \eqref{the-int}    is arbitrarily small for   large  $M\le \ \ov{M}.$

This ends the proof that $(Y^M,Z^M,U^M)_M$ is a Cauchy sequence in $\mathcal{S}^2\times L_2(W)\times L_2(\tilde{\nN}).$ 
 We call the limit  $(Y,Z,U)$.

From Proposition \ref{bounds-independend-from-M} we conclude that $Z$ and $U$ satisfy \eqref{bd-cond-x} setting $a:=a_\infty$ and $b:=b_\infty$.
The bound for $Y$ follows  in the same way from the convergence of $Y^M$ to $Y$  in $\mathcal{S}^2$ and the bound in  \eqref{Y-pi-bound}.  \bigskip \\
Looking at \eqref{BSDE-truncated}, we get that the left hand side $Y_t^M$ converges in $L^2$ to $Y_t$, and the stochastic integrals $\int_t^T Z^M_sdW_s, \int_{]t,T]\times\RR}U^M_s(x)\tilde{\nN}(ds,dx)$ converge in $L^2$ to the respective terms $\int_t^T Z_sdW_s, \int_{]t,T]\times\RR}U_s(x)\tilde{\nN}(ds,dx)$. It remains to show that also $I^M:=\int_t^Tf^M_{h}(s,X_s,Y^M_s,Z^M_s,U^M_s)ds$ does converge.
 
Since $I^M$ is a Cauchy sequence in $L^2$ it converges in $L^2$. There exists a subsequence s.t.~$I^{M_k}$ converges a.s. Moreover, the estimates \eqref{Z-and-U-pi-bound} and  
\eqref{Y-pi-bound}   imply that 
\begin{align*}
  f^M_h(s,X_s,\Theta^M_s) \le C(1+ |X|_\infty^{1+r})
\end{align*}
which is an integrable bound. Then by the dominated convergence theorem 
\begin{align*}
\lim_{k\rightarrow +\infty} I^{M_k}=\int_t^T f_h \Big(s,X^s,\Theta_s\Big)ds.
\end{align*}
\bigskip

Finally, to get that the solution is Malliavin differentiable, we may apply Lemma \ref{lemma123} using that we have $L^2$ convergence of $(Y^M_s,Z^M_s, U^M_s)$ to $(Y_s,Z_s, U_s).$ 
For $x \neq 0$ the uniform bound for the   $\mathbb{D}_{1,2}^{\R_0}$ norms  follows immediately from   Proposition \ref{bounds-independend-from-M}   and Lemma \ref{derivative-bounds}.  \\

To get a uniform bound for the  $\mathbb{D}_{1,2}^0$  norms   if $x=0$ we reuse  the proof of Lemma \ref{lem:apriori}: 
We set
\begin{align*}
&(Y,Z,U)  :=  (0,0,0), \quad (Y',Z',U')  := (D_{s,0}Y^M, D_{s,0}Z^M,D_{s,0}U^M), \\
&\xi := 0, \quad \xi' :=D_{s,0} g^M(X),\quad f = 0, \quad
f_h' = D_{s,0}  f^M \left( t, X^t, \Theta^M_t  \right).
\end{align*}
Clearly, the assumptions on $Z$ and $U$ of Lemma \ref{lem:apriori} are not satisfied for  $D_{t,0}Z^M$ and $D_{t,0}U^M$, but it is possible to derive the a priori estimate as follows: Note that   Lemma \ref{chain-rule-lemma} implies that 
\begin{align*}
& |D_{s,0}  f^M \left( t, X^t, \Theta^M_t  \right) |\\& \le |  ( D_{s,0}  f^M ( t, X^t,y) ) |_{y=(Y^M_t,Z^M_t,H^M_t)} |\notag\\
  &\quad  + L_y \,| D_{s,0} Y^M_t|  +   (c+ |Z^M_t|^{\tr} ) |D_{s,0} Z^M_t | +  (c+ |H^M_t|^{\tr} ) |D_{s,0} H^M_t |. 
\end{align*}
Using  \eqref{Diff-BSDE}, the counterpart of \eqref{Y-timesDiff-f} reads as 
\begin{align*}
&|   D_{s,0}  Y^M_t| | D_{s,0}  f^M\left( t, X^t, \Theta^M_t  \right) | \\
&\le |   D_{s,0}  Y^M_t| | D_{s,0}  f^M\left( t, X^t, \Theta^M_t  \right) -D_{s,0}  f^M(t,X^t, 0) |\\
& \quad + |   D_{s,0}  Y_t| |D_{s,0}  f^M(t,X^t, 0) | \\
 &\leq C(| D_{s,0}  Y^M_t|^2 +| D_{s,0}  Y^M_t| \, |D_{s,0}  f^M(t,X^t, 0) |)\nonumber\\
  &\quad+C| D_{s,0}  Y^M_t|(1+ |X^s|_\infty^{ \tr l}  )  ( |D_{s,0} Z_t^M|+\|D_{s,0}U^M_t\|_{L^2(\nu)})\one_{Q^c_{n,s}} \nonumber\\
  &\quad + C|D_{s,0}  Y_t^M|(1+|X^s|_\infty^{{\tt r+1}})\one_{Q_{n,s}}.
\end{align*}
From here we can just proceed like in the proof of  Lemma \ref{lem:apriori}
to get 
\begin{align}\label{apriori_est-2}
  &\EE|D_{s,0} Y^M|_\infty^2 + \EE\int_0^T|D_{s,0} Z^M_t|^2dt +\EE\int_0^T  \| D_{s,0} U^M_t\|_{L^2(\nu)}^2dt\notag \\
  &\leq Ce^{Cn}\bigg(\EE|D_{s,0} g^M(X)|^2+\EE\bigg(\int_0^T| D_{s,0} f^M(t,X^t,0)|dt\bigg)^2\bigg) \notag\\
  &\quad+e^{-n}C\EE \left[e^{(T+1)^2|X|_\infty)}\right].
   \end{align}
This provides us with a uniform bound  for the  $\mathbb{D}_{1,2}^0$  norms which can be derived using \eqref{Df-estimate} combined with Lemma \ref{exponential-bound}   so that thanks to the convergence in $L^2$ of $(Y^M_t, Z^M_t, U^M_t)$ to $(Y_t, Z_t, U_t)$ 
we get by Lemma \ref{lemma123}   that  $Y_t, Z_t, U_t(v) \in \mathbb{D}_{1,2}^0$ and also   $\int_t^T Z_{s}   dW_s $  and
$\int_{{]t,T]}\times{\R_0}}U_{s}(v) \tilde \nN(ds,dv)$ are in  $\mathbb{D}_{1,2}^0.$  Then we know that also 
$\int_t^T  f\left( s, X^s,Y_s, Z_s,  H_s  \right)ds \in \mathbb{D}_{1,2}^0.   $  
 \qed

\appendix \section{Appendix}\label{app}

\subsection{Stochastic Gronwall and Bihari-LaSalle inequalities}

The article \cite{SarahGeiss21} contains stochastic Gronwall and Bihari-LaSalle inequalities for many different settings. For the convenience of the reader we cite one special  case  
here  (\cite[Theorem 3.1(a)] {SarahGeiss21})  which suits our situation.
\begin {thm} \label{Sarah2}
Assume that $( {\sf X}_t)_{t \ge 0}$ is  an adapted, c\`adl\`ag, non-negative process  satisfying the 
inequality 
\begin{align}  \label{x-inequality2}
  {\sf X}_t  \le \int_{]0,t]}  \eta\big(\sup_{0 \le u < s} {\sf X}_{u}\big)  dA_s  + M_t + H_t,
\end{align}
where for $T>0$ and  $0<p<1$ it holds
\begin{itemize}
\item There is a $c_0 \ge 0$ such that  $ {\sf X}_t \ge c_0$ for  all $t\ge 0$. 
\item $\eta\colon]c_0,\infty[\to]0,\infty[$ is continuous, non-decreasing and convex and such that
$$ \eta^{(p)}(x):= \frac{p}{1-p}\eta(x^{1/p})x^{1-1/p} $$ 
is convex and $C^1,$    and   $ \lim_{x\searrow c_0}\eta^{(p)}(x):=0.$
\item $( A_t)_{t \ge 0}$  is a  predictable, non-decreasing c\`adl\`ag process with $A_0 = 0$.

\item  $( H_t)_{t \ge 0}$ is an adapted,  non-negative non-decreasing c\`adl\`ag process, \\ $\E H_T < \infty$
\item   $( M_t)_{t \ge 0}$ is a c\`adl\`ag local martingale with $M_0 = 0.$
\end{itemize}
Then one has for $0<p<1$
\begin{align} \label{Bihari}  
\E_0 \big[ G^{-1}\big(G(|{\sf X^T}|_\infty)-(1-p)^{-1}A_T \big)^p\big]  \le   \frac{1}{(1-p)}    (\E_0 [  H_T ] )^p   
\end{align}
where $G(x):=\int_r^x\frac{1}{\eta(u)}du$ for some $r>c_0$. \bigskip

 In the special case  where $\eta(x) =x$ and $( A_t)_{t \ge 0}$ is deterministic we have for $0<p<1$
\begin{align} \label{stochGronwall} 
   \E_0 [ |{\sf X^T}|_\infty^p \, ]   \le \frac{1}{1-p} \,   (\E_0 [  H_T ] )^p \, e^{\frac{p}{1-p} A_T} .
 \end{align}  
\end{thm}

 \subsection{Proof of Lemma \ref{chain-rule-lemma} \label{proof-of-chain-rule}}
We will use Sugita's approach to define $\mathbb{D}_{1,2}$ which uses some kind of G\^ateaux  derivative.  We need some notation:
First of all, we exploit the L\'evy-It\^o decomposition, especially, the independence of the Gaussian term and the jump term, to interpret  the probability space 
$(\Omega,\mathcal{F},\mathbb{P})$ as  the completion of  $(\Om^W\times\Om^J, \ftn^W\otimes\ftn^J, \PP^W\otimes\PP^J),$ where 
$(\Om^W, \ftn^W, \PP^W)$ is the canonical Wiener space and the separable Hilbert space $E:=(\Om^J,\ftn^J,\PP^J)$ carries the pure jump process (for more details see  \cite{GeissStein16}). 
Following Janson (\cite[Example 15.6]{Janson}) we let
$$H:=\Big \{ \int_0^T h(s) dW_s: h \in L^2[0,T] \Big  \}$$
be the Hilbert space with the  inner product
$\langle \eta_1,  \eta_2 \rangle := \E \eta _1 \eta_2 .$
For $h \in L^2([0,T])$, we  
  define
$$  g_h(t):= \int_0^t h(s)ds $$
and the {\it Cameron-Martin shift}   
$$ \rho_h  \xi(\om^W):= \xi( \om^W +  g_h)$$
 for any random variable $\xi \in L^0(\PP^W; E).$
 We will denote by $HS(L^2([0,T]);E)$ the space of Hilbert-Schmidt operators between $L^2([0,T])$ and $E.$     
\begin{definition} [\cite{Bogachev,Janson}]
\begin{enumerate}[(i)]
\item A random variable $\xi \in  L^0(\PP^W; E)$  is {\it absolutely continuous along  $h \in L^2([0,T]) $} ($h$-a.c.)  if there exists a  random
variable $ \xi^h \in  L^0(\PP^W; E)$     such that
$ \xi^h = \xi \,\,\, a.s.$ and for all $\om^{_W\!\!}  \in \Om^W$ the map  \[ u \mapsto \xi^h(\om^{_W\!\!}  + u \, g_h)  \]
 is absolutely continuous  on bounded intervals of $\R.$
 \item $\xi \in L^0(\PP^W; E)$ is {\it ray absolutely continuous} (r.a.c.) if $\xi$ is $h$-a.c. for every $h \in L^2([0,T]).$
\item For  $\xi \in L^0(\PP^W; E)$  and $h \in L^2([0,T])$ we say the {\it directional derivative  $\partial_h \xi \in L^0(\Om^W; E)$ exists } if
\equa
\frac{ \rho_{u h} (\xi) - \xi}{u} \to^{\hspace*{-0.7em}\PP^W } \,    \partial_h \xi,  \quad u \to 0.
\tion
\item $\xi \in L^0(\PP^W; E)$ is called  {\it stochastically G\^ateaux differentiable} (s.G.d.)    if  $\partial_h \xi$ exists
for every $h \in L^2([0,T])$ and there exists an  $HS(L^2([0,T]);E)$-valued random variable denoted by $\tilde \D \xi$ such that for every $h \in L^2([0,T])$
\equa
        \partial_h \xi=   \langle \tilde \D \xi ,h\rangle_{L^2([0,T])}, \quad   \PP^W\text{-} a.s.
\tion
\end{enumerate}
 \end{definition}

 The next theorem does in fact hold for $p\ge 1$ while we only need here the version for $p=2$.
 \begin{thm} [  {\cite[Theorem 3.1]{Sugita}}] 
\label{Sugita-thm}
 Then   $\mathbb{D}_{1,2}^0$ can be identified with
\begin{align*}
\mathbb{D}^W_{1,2}(E)  := \!\{\xi \in L^2(\PP^W\!; E)\colon \! \xi &\text{ is  r.a.c., s.G.d. and }\\ & \tilde\D \xi \in L^2(\PP^W\!\!; H\!S(L^2([0,T]);\!E)) \},
\end{align*}
and  for  $\xi \in  \mathbb{D}^W_{1,2}(E) $  it holds  $D_{\cdot,0} \xi = \tilde \D \xi$ a.s.
\end{thm}

\bigskip
\begin{proof}[Proof of  Lemma \ref{chain-rule-lemma}]
We start with item \eqref{X-estimate}. Since 
\begin{align*} 
 \rho_{u h} X_t -X_t = \int_0^t b(r,   \rho_{u h}X^r)   - b(r,   X^r)  dr + u \int_0^t \sigma(r) h(r)  dr. 
\end{align*} it holds
\begin{align*} 
 |\rho_{u h} X_t -X_t|   &\le  L_b \int_0^t | \rho_{u h}X^r  -  X^r |_\infty dr + u K_\sigma   \int_0^t | h(r) | dr,
\end{align*}
so that  by Gronwall's inequality,
\begin{align} \label{Cam-Martin-shift-of-X}
 |\rho_{u h} X^t -X^t|_\infty \le  u K_\sigma \| h \|_{L^1[0,T]}  \,\, e^{L_b t}.
\end{align}
From Proposition  \ref{bounds-for-DX}  we know that  $X_t$  is Malliavin differentiable.
Now  \cite[Theorem 3.11.]{GeissStein16} implies that in probability,
\begin{align*}   \langle  D_{\cdot,0} X_t ,h\rangle_{L^2([0,T])} 
&\le   \lim_{u \downarrow 0}  \frac{ | \rho_{u h} X_t - X_t |  }{u}   \le   K_\sigma \| h \|_{L^1[0,T]}  \,\, e^{L_b t},
\end{align*}
for any $h \in L^2([0,T])$ with $\int_0^T |h(r)|dr \le1.$
Hence   for a.e. $s \in [0,t],$
$$|D_{s,0} X_t |  \le  K_\sigma  \,\, e^{L_b t}.$$
In the same way one can prove  item \eqref{f-estimate}: It holds that  $\f  \left( t, X^t, (Y_1, ..., Y_d)\right ),$ $ \f  \left( t, \bx^t, (Y_1, ..., Y_d)\right ),\f  \left( t, X^t, (y_1, ..., y_d)\right ) \in \mathbb{D}_{1,2}^0. $   This follows  by Theorem \ref{StefanXilin}, for example  for the first term, from  the estimate
\begin{align*} 
&  \| \f  \left( t, X^t, Y \right ) - \f  \left( t, X^{t,\varphi},  Y^\varphi\right )\|_{L^2} \\
& \le  L_\bx \,\, \|  \Big (c+ \frac{\beta}{2}(|X^t|^ {\tt r}_\infty + |X^{t,\varphi}|^ {\tt r}_\infty) \Big)   |X^t-X^{t,\varphi}|_\infty \|_{L^2}  \\
  & \quad  + \sum_{k=1}^d  L_{y_k}  \, \left  \| \, Y_k- Y_k^\varphi \right\|_{L^2}  \le c \varphi.
\end{align*}
Indeed,  by Theorem \ref{StefanXilin} we  have that  $\sup_{0<\varphi \le 1}  \frac{\left  \| \, Y_k- Y_k^\varphi \right\|_{L^2}}{\varphi} <\infty.$ Moreover, thanks to the Cauchy-Schwarz inequality and \eqref{the-p-estimate},
\begin{align*}  
& \Big \|  \Big (c+ \frac{\beta}{2}(|X^t|^ {\tt r}_\infty + |X^{t,\varphi}|^ {\tt r}_\infty) \Big)   |X^t-X^{t,\varphi}|_\infty  \Big \|_{L^2}\\
& \le( c+ \beta  \||X^t|^ {\tt r}_\infty  \|_{L^4}) \,  \||X^t-X^{t,\varphi}|_\infty \|_{L^4}  \le C \varphi.
\end{align*} 
Especially  (see \cite[Theorem 3.11]{GeissStein16} ) we have that in probability,
\begin{align*}  &\lim_{u \downarrow 0}  \frac{  \rho_{u h} \f  \left( t, X^t, (Y_1, ..., Y_d)\right ) -  \f  \left( t, X^t, (Y_1, ..., Y_d)\right )}{u} \\
& =  \langle  D_{\cdot,0}  \f  \left( t, X^t, (Y_1, ..., Y_d)\right ) ,h\rangle_{L^2([0,T])}.
\end{align*}
On the other hand, 
\begin{align*}  & |\rho_{u h} \f  \left( t, X^t, y\right ) -  \f  \left( t, X^t, y\right ) | \\
& \le  L_\bx \, \Big ( c+  \frac{\beta}{2}(|\rho_{u h} X^t|^ {\tt r}_\infty + |X^t|^ {\tt r}_\infty) \Big)   \,  | \rho_{u h} X^t -X^t |_\infty \\
& \le    L_\bx \, \Big ( c+  \frac{\beta}{2}(|\rho_{u h} X^t|^ {\tt r}_\infty + |X^t|^ {\tt r}_\infty) \Big)   \,  u K_\sigma \| h \|_{L^1[0,T]}  \,\, e^{L_b t},
\end{align*}
where we used \eqref{Cam-Martin-shift-of-X}.  Hence we get for any $h \in L^2([0,T])$ with $\int_0^T |h(r)|dr \le 1$  
\begin{align*}  \langle  D_{\cdot,0}  \f  \left( t, X^t, y)\right ) ,h\rangle_{L^2([0,T])} \le L_\bx  K_\sigma  
\big ( c+  \beta |X^t|^ {\tt r}_\infty \big) e^{L_b T}.
\end{align*}
This gives 
$$\|  D_{\cdot,0}  \f  \left( t, X^t, y)\right )\|_{L^\infty[0,T]}     \le L_\bx K_\sigma \big ( c+  \beta |X^t|^ {\tt r}_\infty \big) e^{L_b T}. $$
We will show that in probability,
\begin{align*}  &\lim_{u \downarrow 0}  \frac{  \rho_{u h} \f  \left( t, X^t, (Y_1, ..., Y_d)\right ) -  \f  \left( t, X^t, (Y_1, ..., Y_d)\right )}{u} \\
&=  \lim_{u \downarrow 0}  \frac{ \f  \left( t, \rho_{u h}X^t, \rho_{u h}( Y_1, ..., Y_d)\right ) -  \f  \left( t, X^t, \rho_{u h}( Y_1, ..., Y_d)\right )}{u} \\
& \quad+ \sum_{k=1}^d \lim_{u \downarrow 0} \bigg ( \frac{  \f  \left( t, X^t, \rho_{u h}(Y_1, ...,Y_{d+1-k}),..., Y_d\right )}{u} \\
& \quad\quad \quad  \quad \quad \quad - \frac{    \f  \left( t, X^t, \rho_{u h}(Y_1, ...,Y_{d-k}),...,Y_d\right )}{u} \bigg ) \\
& =  \langle  D_{\cdot,0}  \f  \left( t, X^t, y\right ) ,h\rangle_{L^2([0,T])} \mid_{y =Y} \\
&  \quad +  \sum_{k=1}^d  \langle  D_{\cdot,0}  \f  \left( t, \bx^t, y_1,..., Y_k,...,y_d\right ) ,h\rangle_{L^2([0,T])} \mid_{\bx^t =X^t, y_i =Y_i, i\neq k}. 
\end{align*}
From this  \eqref{chain-rule} follows. The existence of the $G_1,...,G_d$ we get from Nualart  \cite[Proposition 1.2.4]{Nualart}.
To  see that the above limits work, we use  \cite[Theorem 14.1 (ix)]{Janson}  which states that the Cameron Martin shift is continuous as a map
$L^2([0,T])\times L^0(\PP^W; E) \to L^0(\PP^W; E) $ given by  $(\eta, \xi) \mapsto  \rho_{\eta} \xi. $
Moreover,  $\lim_{u \downarrow 0}  \rho_{u h}\xi =\xi$   in probability, for any $\xi \in L^0(\PP^W; E).$
Clearly, it holds
  \begin{align*}
& \lim_{u \downarrow 0}  \frac{ \f  \left( t, \rho_{u h}X^t, \rho_{u h}( Y_1, ..., Y_d)\right ) -  \f  \left( t, X^t, \rho_{u h}( Y_1, ..., Y_d)\right )}{u} \\
&= \lim_{u \downarrow 0}  \rho_{u h} \left (  \frac{ \f  \left( t, X^t,(Y_1, ..., Y_d)\right ) -  \f  \left( t,  \rho_{-u h}X^t, ( Y_1, ..., Y_d)\right )}{u} \right ). 
\end{align*}
Now, considering the map $(\eta, \xi) \mapsto  \rho_{\eta} \xi$ for 
 $$  (\eta, \xi):=  \Big (uh,  \frac{ \f  \left( t, X^t,(Y_1, ..., Y_d)\right ) -  \f  \left( t,  \rho_{-u h}X^t, ( Y_1, ..., Y_d)\right )}{u} \Big ) $$
we get the convergence to $\langle  D_{\cdot,0}  \f  \left( t, X^t, y\right ) ,h\rangle_{L^2([0,T])} \mid_{y =Y}$
thanks to  the joint  continuity   and  the convergence in probability 
  \begin{align*}
 & \lim_{u \downarrow 0}   \frac{ \f  \left( t, X^t,(Y_1, ..., Y_d)\right ) -  \f  \left( t,  \rho_{-u h}X^t, ( Y_1, ..., Y_d)\right )}{u}  \\
 &= \langle  D_{\cdot,0}  \f  \left( t, X^t, y\right ) ,h\rangle_{L^2([0,T])} \mid_{y =Y}.
  \end{align*}
\end{proof}

\bibliographystyle{plain}

\end{document}